\documentclass[letterpaper,11pt]{amsart}
\usepackage[margin=1.2in]{geometry}
\usepackage{amsmath,amsthm,amssymb}
\usepackage{xspace,xcolor}
\usepackage[breaklinks,colorlinks,citecolor=teal,linkcolor=teal,urlcolor=teal,pagebackref,hyperindex]{hyperref}
\usepackage[alphabetic]{amsrefs}
\usepackage[all]{xy}
\usepackage[english]{babel}
\usepackage{enumitem}
\usepackage{tikz,xcolor}
\usepackage{tikz-cd}
\usepackage{mathrsfs}
\usepackage{color}

\newcounter{intro}

\newtheorem{intro-conjecture}[intro]{Conjecture}
\newtheorem{intro-corollary}[intro]{Corollary}
\newtheorem{intro-theorem}[intro]{Theorem}

\newcommand{\theoremref}[1]{\hyperref[#1]{Theorem~\ref*{#1}}}
\newcommand{\sectionref}[1]{\hyperref[#1]{Section~\ref*{#1}}}
\newcommand{\lemmaref}[1]{\hyperref[#1]{Lemma~\ref*{#1}}}
\newcommand{\definitionref}[1]{\hyperref[#1]{Definition~\ref*{#1}}}
\newcommand{\propositionref}[1]{\hyperref[#1]{Proposition~\ref*{#1}}}
\newcommand{\conjectureref}[1]{\hyperref[#1]{Conjecture~\ref*{#1}}}
\newcommand{\corollaryref}[1]{\hyperref[#1]{Corollary~\ref*{#1}}}
\newcommand{\exampleref}[1]{\hyperref[#1]{Example~\ref*{#1}}}
\newcommand{\remarkref}[1]{\hyperref[#1]{Remark~\ref*{#1}}}
\newcommand{\itemref}[1]{\hyperref[#1]{Item~\ref*{#1}}}
\newcommand{\equationref}[1]{\hyperref[#1]{Equation~(\ref*{#1})}}
\newcommand{\formularef}[1]{\hyperref[#1]{Formula~(\ref*{#1})}}
\newcommand{\conditionref}[1]{\hyperref[#1]{Condition~(\ref*{#1})}}

\theoremstyle{plain}
\newtheorem{thm}{Theorem}[section]

\newtheorem{lem}[thm]{Lemma}
\newtheorem{prop}[thm]{Proposition}
\newtheorem{cor}[thm]{Corollary}

\theoremstyle{definition}
\newtheorem{defi}[thm]{Definition}

\newtheorem{eg}[thm]{Example}

\theoremstyle{remark}
\newtheorem{rmk}[thm]{Remark}

\def\Mustata{Mus\-ta\-\c{t}\u{a}\xspace}

\def\N{{\mathbf N}}
\def\Z{{\mathbf Z}}
\def\Q{{\mathbf Q}}

\def\C{{\mathbf C}}

\def\A{{\mathbf A}}

\def\cB{\mathcal{B}}

\def\cD{\mathcal{D}}
\def\cE{\mathcal{E}}

\def\cH{\mathcal{H}}
\def\cI{\mathcal{I}}
\def\cJ{\mathcal{J}}
\def\cK{\mathcal{K}}

\def\cM{\mathcal{M}}
\def\cN{\mathcal{N}}
\def\cO{\mathcal{O}}

\def\cQ{\mathcal{Q}}

\def\cT{\mathcal{T}}

\def\fra{\mathfrak{a}}
\def\frb{\mathfrak{b}}

\def\frm{\mathfrak{m}}

\def\.{\cdot}
\def\^{\widehat}

\def\de{\partial}

\def\({\left(}
\def\){\right)}

\renewcommand{\and}{ \ \ \text{ and } \ \ }

\begin{document}

\title[Multiplier modules, $V$-filtrations and B-S polynomials on singular varieties]{Multiplier modules, $V$-filtrations and Bernstein-Sato polynomials on singular ambient varieties}

\author[B.~Dirks]{Bradley Dirks}

\address{Department of Mathematics, Stony Brook University, Stony Brook, NY 11794-3651, USA}

\email{bradley.dirks@stonybrook.edu}
\thanks{The author was partially supported by NSF-MSPRF grant DMS-2303070.}

\subjclass[2020]{14F10, 14B05, 14J17}

\maketitle

\begin{abstract} We show that the relation between multiplier ideals and $V$-filtration on the structure sheaf due to Budur-Musta\c{t}\u{a}-Saito generalizes to singular irreducible varieties, by replacing multiplier ideals with multiplier modules and the structure sheaf with the intersection complex Hodge module.

This is applied to a Skoda theorem for such modules as well as a $\mathcal D$-module theoretic proof of Ajit's formula relating the multiplier modules of an ideal to those of the Rees parameter in the extended Rees algebra.

Moreover, we define a Bernstein-Sato polynomial for the pair of a variety and an ideal sheaf on it. We relate the roots to the jumping numbers of the multiplier modules. If the ideal is generated by a regular sequence on a rational homology manifold, we show that the absence of integer roots of the polynomial implies that the subvariety defined by the ideal is a rational homology manifold.
\end{abstract}

\section{Introduction}
For a smooth complex variety $X$, the $V$-filtration of Kashiwara-Malgrange \cites{Kashiwara,Malgrange}  on a holonomic $\cD_X$-module along a function $f\in \cO_X(X)$ is an incredibly important construction in the theory of $\cD$-modules, mixed Hodge modules, and singularities. It is the fundamental construction which allows for a $\cD$-module theoretic incarnation of the nearby and vanishing cycles functors \cite{DeligneVC}. It is also intimately related with the theory of Bernstein-Sato polynomials \cite{SabbahOrder}.

Aside from the relation to the Bernstein-Sato polynomial, the most direct application of $V$-filtrations to the study of singularities is the fundamental result of Budur-Saito \cite{BS}: if $\Gamma \colon X \to X\times \A^1_t$ is the graph embedding and $\Gamma_+ \cO_X$ is the pushforward as a $\cD_X$-module, then $V^\bullet \Gamma_+\cO_X$ induces in a natural way a filtration on $\cO_X$ (due to the definition of the $\cD$-module pushforward, see \exampleref{eg-Kashiwara} below) and we have the following \cite{BS}*{Thm. 1}:
\[ \cJ(X,f^{\alpha-\varepsilon}) = V^\alpha \cO_X \text{ for all } \alpha > 0 \text{ and } 0 < \varepsilon \ll 1.\]

Here the left hand side is the multiplier ideal of $f$ with exponent $\alpha-\varepsilon$, defined through numerical data on a log resolution of the pair $(X,f)$ (a brief review on the definition is in Section \ref{sec-prelim} below). This relation was influential in Saito's definition of microlocal multiplier ideals \cite{SaitoHodgeIdeal} and their comparison to the Hodge ideals of \Mustata-Popa \cite{MP2}. Another proof of this equality is given in the forthcoming \cite{DY}*{Lem. 5.9}.

In fact, the same is true for an ideal sheaf: if $\fra = (f_1,\dots, f_r) \subseteq \cO_X$ is an ideal sheaf, with $X$ smooth, then \cite{BMS}*{Thm. 1} shows that
\[ \cJ(X,\fra^{\alpha-\varepsilon}) = V^\alpha \cO_X \text{ for all } \alpha >0 \text{ and } 0 < \varepsilon \ll 1,\]
where now the $V$-filtration on $\cO_X$ is defined by the $V$-filtration on $\Gamma_+ \cO_X$, and $\Gamma \colon X \to X\times \A^r_t$ is the graph embedding along $f_1,\dots, f_r$. As above, $\cJ(X,\fra^{\alpha-\varepsilon})$ is the multiplier ideal of $\fra$ with exponent $\alpha - \varepsilon$, whose definition will also be reviewed in Section \ref{sec-prelim} below. It is an important singularity invariant of the pair $(X,\fra)$ corresponding to singularity classes in the minimal model program.

The main goal of this note is to explain how this result also holds for singular (reduced and irreducible) $X$, with appropriate modifications. To pass to the singular case, even the statement requires the theory of mixed Hodge modules, which we review briefly in Section \ref{sec-prelim} below. As $X$ is singular, we replace $\cO_X$ (which need not underlie a mixed Hodge module) with ${\rm IC}_X^H$, the intersection complex Hodge module on $X$. However, as the $V$-filtration is only defined for $\cD$-modules on smooth varieties, we will have to make use of an auxiliary embedding of $X$ into a smooth variety $Y$, though we will see that the result doesn't depend on this embedding.

Let $X$ be a reduced, irreducible variety. Given an embedding $i\colon X \to Y$, we can define $i_* {\rm IC}_X^H$ a pure Hodge module on $Y$ with underlying filtered \emph{right} $\cD_Y$-module denoted $(\cM,F)$. We assume $f_1,\dots, f_r \in \cO_X(X)$ are restricted from $Y$, which is true up to replacing $X$ by an open subset. Then we can define the graph embedding $\Gamma \colon Y \to Y \times \A^r_t$ and the module $\Gamma_+ \cM$ admits a $V$-filtration along $t_1,\dots, t_r$. We index our $V$-filtrations on right $\cD$-modules following the convention for left $\cD$-modules (see the discussion in Section \ref{sec-prelim} below), to more closely match the multiplier ideal conventions.

In the above, the results were stated using multiplier ideals. However, as we do not want to impose singularity conditions on $X$, we will state our main result with the \emph{multiplier module} $\cJ(\omega_X,\fra^{\lambda-\varepsilon})$, which is well-defined for any non-zero ideal sheaf on a complex algebraic variety. The definition of this module will be reviewed in Section \ref{sec-prelim} below.

The main result is the following:
\begin{thm} \label{thm-main} Assume $X$ is reduced and irreducible. For $(\cM,F)$ the filtered right $\cD_Y$-module underlying ${\rm IC}^H_X$ and $\Gamma \colon Y \to Y \times \A^r_t$ the graph embedding along generators $f_1,\dots, f_r$ of the non-zero ideal $\fra$, we have for all $\lambda > 0$ the equality
\[ \cJ(\omega_X,\fra^{\lambda-\varepsilon}) = F_{-\dim X} V^\lambda \Gamma_+ \cM \text{ for all } 0 < \varepsilon \ll 1.\]
\end{thm}

\begin{rmk} \label{rmk-wellDefined} The left hand side of the above isomorphism is independent of the embedding $X\subseteq Y$. The right hand side can also be viewed as independent of the choice of embedding $X\subseteq Y$. Indeed, though to have a filtered $\cD$-module underlying a mixed Hodge module one must work in an ambient smooth variety, the first non-zero piece of the Hodge filtration is a well-defined $\cO_X$-coherent sheaf which does not depend on choice of embedding. In the case of ${\rm IC}_X$, it is the Grauert-Riemenschneider sheaf $\pi_*(\omega_X) = \omega_X^{\rm GR}$ \cite{SaitoMHC}*{Pg. 3}.

We will see directly in \lemmaref{lem-unambiguous} and \lemmaref{lem-ChangeY} that the right hand side of the equality is well-defined.
\end{rmk}

Our proof follows closely that of \cite{BMS}*{Thm. 1} but with a different ingredient to handle the fact that $X$ is not smooth. The same proof would go through verbatim if $X_{\rm sing} \subseteq V(f_1,\dots, f_r)$ as mentioned in Remark \ref{rmk-SameProof} below, but that is too restrictive a hypothesis.

In our theorem, we have no assumption on the singularities of $X$, which explains why we had to work with the multiplier module in place of the multiplier ideal. If $X$ were Gorenstein, one could obtain a formula for the multiplier ideal $\cJ(X,\fra^{\lambda-\varepsilon})$ in a similar way.

Using this and properties of the $V$-filtration, we can conclude the following statements (which can be proven by classical methods as well). The second is a Skoda type statement.

\begin{cor} \label{cor-Skoda} If $\lambda$ is a jumping number of $\cJ(\omega_X,\fra^{\bullet})$, then so is $\lambda + j$ for all $j\in \mathbf N$.

Also, we have
\[ \cJ(\omega_X,\fra^{\lambda}) = \fra \cdot \cJ(\omega_X,\fra^{\lambda-1}) \text{ for all } \lambda > r,\]
where $r$ is the number of defining functions for $\fra$.

Finally, 
\[{\rm depth}(\fra,\omega_X^{\rm GR}) \geq \lceil {\rm lct}(\omega_X,\fra)\rceil.\]
\end{cor}

In the statement above, ${\rm depth}(\fra,\cM) = \min_{\fra \subseteq \frm} {\rm depth}(\fra_{\frm},\cM_{\frm})$ for a coherent $\cO_X$-module $\cM$, where the minimum is taken over maximal ideals $\frm \subseteq \cO_X$ containing $\fra$.

\begin{rmk} As communicated to the author by Sung Gi Park, the depth inequality above is immediate when $\omega_X^{\rm GR}$ is a Cohen-Macaulay $\cO$-module (for $X$ normal, this is equivalent to having rational singularities). Indeed, in that case, ${\rm depth}(\fra, \omega_X^{\rm GR}) = {\rm codim}_X({\rm Supp}(\omega_X^{\rm GR} \otimes_{\cO_X} (\cO_X/\fra))) \geq {\rm codim}_X(V(\fra))$ (for the equality, see \cite{Stacks}*{Lem. 47.11.4}). But by restricting to a dense open subset $U\subseteq X$ such that $X,V(\fra)_{\rm red}$ are both smooth, we see that we always have the inequality
\[ {\rm codim}_X(V(\fra)) = {\rm lct}(\omega_U,\sqrt{\fra\vert_U}) \geq {\rm lct}(\omega_X,\fra).\]
\end{rmk}

\theoremref{thm-main} has the following interesting corollary giving a restriction relation between the Du Bois complexes. We use the invariant ${\rm HRH}(X)$ from \cite{DOR}, which measures a partial Poincar\'{e} duality property of $X$. For the precise definition, see \definitionref{defi-HRH} below. Note that the condition ${\rm HRH}(X) \geq 0$ implies $\underline{\Omega}_X^0$ is maximal Cohen-Macaulay as an object in $D^b_{\rm coh}(\cO_X)$, and in fact is equivalent to the quasi-isomorphism $\omega_X^{\rm GR} \cong \mathbb D(\underline{\Omega}_X^0)[-\dim X]$, where $\mathbb D(-) = R\cH om_{\cO_X}(-,\omega_X^\bullet)$ is the Grothendieck duality functor.

\begin{cor} \label{cor-DualDuBois} Let $X$ be reduced and irreducible satisfying ${\rm HRH}(X) \geq 0$ and let $\fra \subseteq \cO_X$ be a coherent ideal sheaf generated by a regular sequence of length $r$ and defining $Z\subseteq X$ so that $i \colon Z \to X$ is a regular embedding of pure codimension $r$ and $Z$ is reduced.

Then $\underline{\Omega}_Z^0$ is maximal Cohen-Macaulay as an object in $D^b_{\rm coh}(\cO_Z)$, and we have
\[  \mathbb D(\underline{\Omega}_Z^0)[-\dim Z] \cong \frac{\cJ(\omega_X,\fra^{r-\varepsilon})}{\fra \cdot \cJ(\omega_X,\fra^{r-1-\varepsilon})} \text{ for } 0 < \varepsilon \ll 1.\]

In particular, if ${\rm lct}(\omega_X,\fra) = r$, then we have an isomorphism for any resolution of singularities $\pi \colon Y \to X$:
\[ \underline{\Omega}_Z^0 \cong Li^* R\pi_*\cO_Y.\]
\end{cor}

The condition on ${\rm HRH}(X)$ above is automatic if $X$ has rational singularities or is a rational homology manifold. This recovers the well-known implication that log canonical singularities are Du Bois \cite{LogCanonicalAreDB,KovacsSchwedeSmith,ParkDB}. Note that we do not take a log resolution of the pair $(X,Z)$ in the corollary statement.

\begin{cor} If $(X,Z)$ is a pair such that $X$ has rational singularities, $i \colon Z \to X$ is a regular embedding of pure codimension $r$ and ${\rm lct}(X,Z) = r$, then $Z$ has Du Bois singularities.
\end{cor}

As another corollary, the main theorem recovers a recent result of Ajit computing the multiplier ideals of $\fra$ in terms of multiplier ideals on the deformation to the normal cone. The corresponding relation of $V$-filtrations (motivated by Verdier's specialization construction \cite{Verdier}) is given in \cite{BMS}*{(1.3.1)}. It is one of the main tools in the study of higher codimension $V$-filtrations.

As above, let $f_1,\dots, f_r$ be functions on smooth $Y$ restricting to functions on the singular subvariety $X \subseteq Y$. Let $T = Y \times \A^r_t$ be the target of the graph embedding morphism. We define the deformation to the normal cone via the extended Rees algebra construction: we let
\[ \widetilde{X} =  {\rm Spec}(\bigoplus_{n\in\Z} \fra^n \cO_X u^{-n}),\]
\[ \widetilde{Y} = {\rm Spec}(\bigoplus_{n\in \Z} \fra^n \cO_Y u^{-n}),\]
\[ \widetilde{T} = {\rm Spec}(\bigoplus_{n\in \Z} (t_1,\dots, t_r)^n \cO_T u^{-n}) \cong Y \times \A^r_z \times \A^1_u,\]
where in the last isomorphism we have $z_i = \frac{t_i}{u}$. We have the commutative diagram where each square is Cartesian:
\[ \begin{tikzcd} \widetilde{X} \ar[r] \ar[d] & \widetilde{Y} \ar[r] \ar[d] & \widetilde{T} \ar[d] \\ X \ar[r] & Y \ar[r]  & T \end{tikzcd}.\]

Of course, we have the Hodge module ${\rm IC}_{\widetilde{X}}$ on $\widetilde{T}$. We also have the module, which we denote $\widetilde{\rm IC}_X$, defined by applying the Verdier specialization process to ${\rm IC}_X$ (see the discussion after the proof of \theoremref{thm-main} below). These modules agree upon restriction to $\{u \neq 0\}$, and so their $V$-filtrations are comparable in positive degrees (see \lemmaref{lem-PositiveVFilt} below). From this observation, the main theorem, and \cite{BMS}*{(1.3.1)}, we conclude the following, which should be compared with \cite{Ajit}*{}:
\begin{cor} \label{cor-Ajit} In the above notation, we have for all $\lambda >0$ an equality of graded submodules
\[ \cJ(\omega_{\widetilde{X}},u^\lambda) = \bigoplus_{k \in \Z} \cJ(\omega_X,\fra^{\lambda-k-1}) u^kdu \subseteq j_*j^*(\omega_{\widetilde{X}}^{\rm GR}) = j_*( \omega_X^{\rm GR} \boxtimes \C[u^{\pm 1}] du),\]
where $j\colon X \times \mathbf G_{m,u} \to \widetilde{X}$ is the inclusion of the complement of $\{u=0\}$.
\end{cor}

Finally, we define a theory of ``Bernstein-Sato polynomials'' for pairs $(X,\fra)$, which is motivated by the construction in \theoremref{thm-main}. We only assume $X$ is connected and reduced in the definition. To any such pair $(X,\fra)$, we associate a polynomial $b_{(X,\fra)}(s)$ defined via the lowest Hodge piece of the intersection Hodge module ${\rm IC}_X^H$ (see Definition \ref{defi-BS} below). These polynomials agree with the usual ones, defined in \cite{BMS}, if $X$ is smooth.

The main properties of this polynomial are the following:
\begin{thm} \label{thm-bFunction} Let $X$ be a connected variety with a finite set $I$ such that $\{X_i \mid i \in I, \dim X_i = \dim X\}$ is the collection of maximal dimensional irreducible components of $X$ and let $(f_1,\dots, f_r) = \fra \subseteq \cO_X$ be an ideal sheaf. Then
\begin{enumerate} \item \label{itm-LCM} We have equality
\[ b_{(X,\fra)}(s) = {\rm lcm}_{i\in I} b_{(X_i,\fra\vert_{X_i})}(s).\]

\item \label{itm-NegRoots} If $\gamma$ satisfies $b_{(X,\fra)}(-\gamma) = 0$, then $\gamma \in \Q_{\geq 0}$. We have $s \mid b_{(X,\fra)}(s)$ if and only if $X_i \subseteq V(\fra)$ for some $i\in I$, in which case $s$ has multiplicity 1 inside $b_{(X,\fra)}(s)$.

\item \label{itm-JumpingNumbers} Assume $X$ is reduced and irreducible. We have $\min \{\lambda \mid b_{(X,\fra)}(-\lambda)= 0\} = \max\{\lambda \mid \cJ(\omega_X, \fra^{\lambda-\varepsilon}) = \omega_X^{\rm GR}\}$, or in other words, the negative of the largest root of $b_{(X,\fra)}(s)$ gives the \emph{log canonical threshold} of the pair $(\omega_X,\fra)$. In fact, we have containment
\[ \{ \lambda \mid \cJ(\omega_X,\fra^{\lambda-\varepsilon}) \neq \cJ(\omega_X,\fra^{\lambda})\} \cap [{\rm lct}(\omega_X,\fra) ,{\rm lct}(\omega_X,\fra) +1)  \subseteq \{\lambda \mid b_{(X,\fra)}(-\lambda) = 0\}.\]

\item Let $\fra\vert_{X_{\rm reg}} = (\widetilde{f}_1,\dots, \widetilde{f}_r) \subseteq \cO_{X_{\rm reg}}$. Then
\[ b_{\widetilde{f}}(s) \mid b_{(X,\fra)}(s).\]
\end{enumerate}
\end{thm}

Property \ref{itm-JumpingNumbers} recovers a well-known result due to Lichtin \cite{Lichtin} and Koll\'{a}r \cite{KollarPairs} and Ein-Lazarsfeld-Smith-Varolin \cite{ELSV} in the hypersurface case.

The following theorem uses the characterization of the rational homology manifold property via purity of the module $\Q^H_Z[\dim Z]$.

\begin{thm} \label{thm-BSRHM} Assume $X$ is a rational homology manifold, $\fra$ is (locally) generated by a regular sequence of length $r$ and that the only integer root of $b_{(X,\fra)}(s)$ is $-r$. Then $Z = V(\fra)$ is a rational homology manifold.
\end{thm}

\begin{rmk} Bernstein-Sato polynomials and $V$-filtrations were defined for certain singular rings $R$ in \cites{BFunctionSing,VFiltSing}. The point of view taken in those papers is to use the ring of differential operators on the ring $R$ itself (so that $R$ is a $\cD$-module, contrary to our setting). 

It would be interesting to understand in which cases the polynomials are comparable. By \exampleref{eg-Node} below, if $X = V(xy) \subseteq \A^2$ and $f=x$, then our polynomial agrees with theirs \cite{VFiltSing}*{Eg. 2.19}. However, for $X = V(x^2+y^3)$, our polynomial has strictly negative roots, but theirs has a positive root of $\frac{1}{2}$ \cite{VFiltSing}*{Eg. 2.20}.
\end{rmk}

\medskip

\noindent\textbf{Outline.} Section \ref{sec-prelim} reviews the definition of multiplier ideals and modules, the definition and properties of $b$-functions and $V$-filtrations of $\cD$-modules and the relevant notions from the theory of mixed Hodge modules. 

Section \ref{sec-proofs} contains the definition of the Bernstein-Sato polynomial $b_{(X,\fra)}(s)$ and proofs of the main results. 

Section \ref{sec-examples} contains some comments about weighted homogeneous examples.
\medskip

\noindent {\bf Acknowledgments} The author thanks Qianyu Chen, Mircea Musta\c{t}\u{a}, Sebasti\'{a}n Olano, Sung Gi Park and Claude Sabbah for conversations related to this work. He also thanks Rahul Ajit for discussions concerning the formula in \corollaryref{cor-Ajit} which led the author to write this note.

\section{Preliminaries} \label{sec-prelim}
In this section, we review the definition of multiplier ideals/modules as well as the pertinent aspects of the theory of $\cD$-modules and mixed Hodge modules. We refer the reader to \cite{LazII}*{Ch. 9} for multiplier ideals, to \cite{HTT} for the theory of $\cD$-modules and to \cite{SaitoMHP,SaitoMHM} for mixed Hodge modules (as well as Schnell's survey \cite{Schnell}).

\subsection{Multiplier ideals and modules} For $Y$ a smooth variety and $f \in \cO_Y(Y)$, the multiplier ideal $\cJ(Y,f^\lambda)$ is defined through a log resolution of the pair $(Y,f)$, meaning a projective birational morphism $\pi \colon Y' \to Y$ such that $Y'$ is smooth and ${\rm div}(\pi^*(f)) + {\rm Exc}(\pi)$ is an SNC divisor on $Y'$. Specifically, we have
\[ \cJ(Y,f^\lambda) = \pi_* \cO_{Y'}(K_{Y'/Y} -\lfloor \lambda {\rm div}( \pi^*(f)) \rfloor),\]
where $K_{Y'/Y}$ is the relative canonical divisor class.

More generally, for an ideal $\fra \subseteq \cO_Y$ and a log resolution $\pi \colon Y'\to Y$ of the pair $(Y,\fra)$, meaning $\pi$ is projective birational, $Y'$ is smooth, and $\fra \cO_{Y'} = \cO_{Y'}(-F)$ for some effective divisor $F$ with the property that $F+ {\rm Exc}(\pi)$ is an SNC divisor on $Y'$, we define
\[ \cJ(Y,\fra^{\lambda}) = \pi_* \cO_{Y'}(K_{Y'/Y}  - \lfloor \lambda F\rfloor).\]

The same definition works if $Y$ is ($\Q$-)Gorenstein. If $Y$ is replaced by a singular irreducible variety $X$, the difficulty in extending this definition lies in the definition (or lack thereof) of the canonical divisor $K_X$ on $X$. There are workarounds to this if one assumes $X$ is normal \cite{dFHacon}, but an alternative is to work instead with the \emph{multiplier module} of the ideal $\fra \subseteq \cO_X$. This is a sub-module of the Grauert-Riemenschneider sheaf $\pi_* \omega_{X'} = \omega_X^{\rm GR}$ defined through a log resolution $\pi \colon X' \to X$ of the pair $(X,\fra)$ by a similar formula to the above:
\[ \cJ(\omega_X,\fra^{\lambda}) = \pi_*\cO_{X'}(K_{X'} - \lfloor \lambda F\rfloor),\]
where $\fra \cO_{X'} = \cO_{X'}(-F)$. As $\lambda$ increases, the multiplier module does not get larger, so this defines a decreasing filtration on $\omega_{X}^{\rm GR}$. The Grauert-Riemenschneider sheaf, multiplier modules and multiplier ideals (when defined) are all independent of the choice of log resolution.

\begin{rmk} By \cite{Shibata}*{Cor. 3.25}, using the definition of multiplier ideal on normal varieties due to \cite{dFHacon} (whose definition we do not give in this paper) we have the containment
\[ \cJ(X,\fra^\lambda) \subseteq (\cJ(\omega_X, \fra^{\lambda}) \colon \omega_X^{\rm GR}),\]
where the right hand side is the colon ideal for the sub-module $\cJ(\omega_X,\fra^{\lambda}) \subseteq\omega_X^{\rm GR}$. In other words,
\[ \cJ(X, \fra^\lambda) \cdot \omega_{X}^{\rm GR} \subseteq \cJ(\omega_X,\fra^{\lambda}),\]
which gives a criterion for triviality of the multiplier module in terms of that of the multiplier ideal.
\end{rmk}

\begin{defi} We define the \emph{log canonical threshold} of the pair $(\omega_X,\fra)$ to be
\[ {\rm lct}(\omega_X,\fra) = \inf \{\lambda \mid \cJ(\omega_X,\fra^{\lambda}) \subsetneq \omega_X^{\rm GR}\}.\]

If $X$ is Gorenstein, then this agrees with the usual notion of log canonical threshold for the pair $(X,\fra)$.
\end{defi}

\subsection{$V$-filtrations on $\cD$-modules}

For $Y$ a smooth variety, the ring $\cD_Y$ of differential operators on $Y$ is the sub-sheaf of rings of $\cE nd_{\C}(\cO_Y)$ generated by $\cO_Y$ (viewed via multiplication on the left) and the tangent sheaf $\cT_Y$. In local coordinates, say $y_1,\dots, y_n$, an arbitrary element of $\cD_Y$ is of the form
\[ P = \sum_{\alpha \in \N^n} h_\alpha \de_y^\alpha,\]
where the sum is finite and $h_\alpha \in \cO_Y$. As it will come up below, we describe the \emph{classical adjoint} involution in this choice of coordinates. Given $P = \sum_\alpha h_\alpha \de_y^\alpha$, we define ${}^t P = \sum_\alpha (-\de_y)^\alpha h_\alpha \in \cD_Y$.

If $\cM$ is a left $\cD_Y$-module (for example, $\cO_Y$), we can define a corresponding right $\cD_Y$-module $\cM^r$ by
\[ \cM^r = \omega_Y \otimes \cM,\]
with action given locally as follows: if $y_1,\dots, y_n$ are coordinates on $Y$, with volume form $dy = dy_1\wedge \dots \wedge dy_n$, then
\[ (dy \otimes m) P = dy \otimes ({}^t P m) \text{ for all local sections } m\in \cM.\]

This operation is invertible, hence gives an equivalence between left and right $\cD_Y$-modules. We will not make use of the inverse transformation. From this construction, we see that $\omega_Y$ is naturally a right $\cD_Y$-module. In the theory of $\cD$-modules, right modules are best behaved for duality and pushforward, and left modules are best behaved for smooth pull-back and comparison with variations of Hodge structure.

As mentioned in the introduction, the $V$-filtration of Kashiwara \cite{Kashiwara} and Malgrange \cite{Malgrange} (the $\Q$-indexed refinement we use is due to Saito) is an incredibly important tool in the theory of $\cD$-modules and mixed Hodge modules. It gives the $\cD$-module analogue of the nearby and vanishing cycles construction, which is compatible with the usual one for constructible complexes under the Riemann-Hilbert correspondence. Saito made incredible use of the $V$-filtration for hypersurfaces in the construction of his theory of mixed Hodge modules.

As we will use a slightly non-standard index convention for $V$-filtrations on right $\cD$-modules, we review the conventions for left $\cD$-modules and simply use the side-changing operation to induce our convention for right $\cD$-modules. For more details, consult \cite{SaitoMHP}*{Sec. 3} for the hypersurface case and \cite{BMS}*{Sec. 1} for the higher codimension case.

First, we define the $V$-filtration on $\cD_Y$ along a smooth subvariety $Z \subseteq Y$. If $\cI_Z$ is the ideal sheaf defining $Z$ in $Y$, then for any $j\in \Z$, we define
\[ V^j_Z \cD_Y = \{ P \in \cD_Y \mid P \cI_Z^k \subseteq \cI_Z^{k+j} \text{ for all } k \in \Z\},\]
though we will drop the $Z$ subscript if it is clear from context.

For example, $\cO_Y \subseteq V^0\cD_Y$ and $\cI_Z \subseteq V^1 \cD_Y$. This filtration is compatible with the ring structure.

Given a coherent left $\cD_Y$-module $\cM$ and a smooth subvariety $Z\subseteq Y$, a $V$-filtration on $\cM$ is a decreasing, discretely and left-continuously $\Q$-indexed filtration $V^\bullet \cM$ with the following properties:
\begin{enumerate} \item We have $V^j\cD_Y \cdot V^\lambda \cM \subseteq V^{\lambda+j}\cM$ for all $j\in \Z, \lambda \in \Q$, with equality when $j\geq 0$ and $\lambda \gg 0$.
\item The $V^0\cD_Y$-module $V^\lambda \cM$ is coherent for all $\lambda \in \Q$.
\item For any derivation $\theta \in V^0\cD_Y$ such that $\theta$ acts as the identity on $\cI_Z/\cI_Z^2$, we have
\[ (\theta - \lambda +r) \text{ is nilpotent on } {\rm Gr}_V^{\lambda}(\cM) = V^\lambda \cM/V^{>\lambda} \cM,\]
where $V^{>\lambda} \cM = \bigcup_{\beta > \lambda} V^\beta \cM$.
\end{enumerate}

If $\cM$ is a left $\cD_Y$-module with $V$-filtration $V^\bullet \cM$ and $\cM^r = \omega_Y \otimes \cM$ is the corresponding right $\cD_Y$-module, then we define
\[ V^\lambda \cM^{r} = \omega_Y \otimes V^\lambda \cM,\]
where we recall that, in local coordinates, the $\cD_Y$-module structure on $\cM^r$ is given by the classical adjoint involution of $\cD_Y$.

\begin{rmk} If $Z \subseteq Y$ is the smooth subvariety defined by a partial system of coordinates $t_1,\dots, t_r$ with vector fields $\de_{t_1},\dots, \de_{t_r}$, then a derivation $\theta \in V^0\cD_Y$ as in the definition of the $V$-filtration can be taken to be $\theta_t = \sum_{i=1}^r t_i \de_{t_i}$.

Hence, we see via the classical adjoint that if $\cM^r = \cM \otimes_{\cO} \omega_Y$, then
\[ (m\otimes \alpha) \cdot \theta_t = ((-\theta_t -r) \cdot m) \otimes \alpha \text{ for all } m\in \cM, \alpha \in \omega_Y\]
where $\theta_t = \sum_{i=1}^r t_i \de_{t_i}$. In this way, in our conventions, we see that
\[ \theta+\lambda \text{ acts nilpotently on } {\rm Gr}_V^\lambda \cM^r,\]
and one could uniquely characterize the $\Q$-indexed $V$-filtration for right $\cD_Y$-modules along $Z$ in a way similar to the characterization given for left $\cD_Y$-modules above.
\end{rmk}

In general, it is quite difficult to describe the $V$-filtration on an arbitrary coherent $\cD$-module $\cM$, as evidenced by the main theorem, where even the lowest Hodge piece contains all of the information of the multiplier module. The following example is the easiest case, and follows essentially by Kashiwara's equivalence.

\begin{eg} \label{eg-Kashiwara} Let $\cM$ be a coherent right $\cD_{Y}$-module supported on $Z\subseteq Y$, where $Z$ and $Y$ are smooth varieties. Then $\cM$ admits a $V$-filtration. Indeed, by Kashiwara's equivalence, $\cM = i_+ \cM_0$ for some $\cM_0$ a coherent $\cD_Z$-module. If $Z = V(y_1,\dots, y_c)$ locally on $Y$, where $y_1,\dots, y_c$ are part of a system of local coordinates with corresponding vector fields $\de_{y_1},\dots, \de_{y_c}$, then the $\cD$-module direct image $i_+(-)$ is defined by
\[ i_+\cM_0 = \bigoplus_{\alpha \in \N^c} \cM_0 \de_y^\alpha\]
where $\cM_0 y_i = 0$ for $1\leq i\leq c$. Then it is not hard to check that
\[ V^\bullet \cM = \bigoplus_{|\alpha| \leq -\bullet} \cM_0 \de_y^\alpha.\]

For example, 
\[ (m \de_y^\alpha)(\theta_y) = (m (\theta_y + |\alpha|) )\de_y^{\alpha} = (m \de_y^{\alpha})|\alpha|,\]
using the fact that $m \theta_y = \sum_{i=1}^r m y_i \de_{y_i} = 0$. Thus, $(m\de_y^{\alpha})(\theta - |\alpha|) = 0$, and from here it is easy to check the other defining properties of the $V$-filtration hold.
\end{eg}

\begin{eg} \label{eg-InvarianceClosedEmbeddings} In a slightly different direction, let $i \colon Y_1 \hookrightarrow Y_2$ be a closed embedding of smooth varieties such that $Z\subseteq Y_2$ is a smooth variety of codimension $r$ in $Y_2$ with $Z' = Z\cap Y_1$ is also smooth of codimension $r$ in $Y_1$.

If $\cM$ is a coherent $Y_1$ module which admits a $V$-filtration along $Z'$ then $i_+ \cM$ admits a $V$-filtration along $Z$. In fact, we will show that the $V$-filtration is essentially invariant by $i_+$.

In local coordinates, assume $Y_1$ is defined inside $Y_2$ by $y_1,\dots, y_c$ and $Z$ is defined inside $Y_2$ by $t_1,\dots, t_r$. Then $Z'$ is defined inside $Y_1$ by $t_1,\dots, t_r$, which form part of a system of coordinates.

By definition of the $\cD$-module pushforward, in this local choice of coordinates we have
\[ i_+ \cM = \bigoplus_{\alpha \in \N^c} \cM \de_y^\alpha.\]

Using the fact that $t_i$ and $\theta_t = \sum_{i=1}^r t_i \de_{t_i}$ commutes with $\de_y^\alpha$, we see without much difficulty that
\[ V^\alpha i_+ \cM = \bigoplus_{\alpha \in \N^c} (V^\alpha \cM)\de_y^\alpha.\]
\end{eg}

By uniqueness of the $V$-filtration, we see that if $\varphi \colon \cM \to \cN$ is a morphism of $\cD_Y$-modules which both admit a $V$-filtration along $Z$, then $\varphi$ is automatically strict with respect to $V^\bullet$. Concretely, this means that the functors $V^\lambda(-)$ and ${\rm Gr}_V^\lambda(-)$ are exact for all $\lambda \in \Q$.

\begin{lem} \label{lem-PositiveVFilt} Let $Z \subseteq Y$ be a smooth subvariety in a smooth variety and let $\cM$, $\cN$ be two coherent $\cD_Y$-modules which admit $V$-filtrations along $Z$. If $U = Y \setminus Z$, let $\varphi \colon \cM\vert_U \to \cN\vert_U$ be an isomorphism of $\cD_U$-modules. Then $\varphi$ induces an isomorphism for all $\lambda > 0$
\[ V^\lambda \cM \cong V^\lambda \cN.\]
\end{lem}
\begin{proof} Note that $\varphi$ induces an isomorphism $\cH^0 j_*(\cM\vert_U) \cong \cH^0 j_*(\cN\vert_U)$, where $j \colon U \to Y$ is the open embedding. By the natural map $\cM \to \cH^0 j_*(\cM\vert_U)$, it suffices to assume that $\varphi$ is induced by restricting a morphism $\psi \colon \cM \to \cN$.

Now, we have the exact sequence
\[ 0 \to \cK \to \cM \to \cN \to \cQ \to 0,\]
where $\cK$ and $\cQ$ are supported on $Z$. By the previous example, $V^\lambda \cK = V^\lambda \cQ = 0$ for all $\lambda > 0$. Thus, the result follows by exactness of $V^\lambda(-)$.
\end{proof}

\subsection{Mixed Hodge modules} As mentioned in the introduction, Saito's theory of mixed Hodge modules is set-up via an extensive use of $V$-filtrations along regular functions. As the main theorem statement requires the Hodge filtration on the pure Hodge module ${\rm IC}_X$, we recall some of the basic notions in this theory.

On a smooth variety $Y$, the data of a mixed Hodge module is a tuple $((\cM,F,W), (\cK,W),\alpha)$ where $(\cM,F)$ is a good filtered $\cD_Y$-module, $W_\bullet \cM$ is a finite increasing filtration by $\cD_Y$-modules, $\cK$ is an algebraically constructible $\Q$-perverse sheaf on the associated analytic space $Y^{\rm an}$ and $\alpha$ is a filtered quasi-isomorphism
\[ \alpha \colon {\rm DR}(\cM,W) \cong (\cK,W)\otimes_{\Q} \C,\]
subject to various restrictions.

Using local embeddings into smooth varieties, one can also define the category of mixed Hodge modules on a singular variety $X$, which we denote by ${\rm MHM}(X)$. It is always an abelian category. An object $M \in {\rm MHM}(X)$ is pure of weight $w$ if ${\rm Gr}^W_j M = 0$ for all $j\neq w$.

\begin{eg} \label{eg-smoothTrivial} If $Y$ is smooth and connected, then $\omega_Y$ underlies a pure Hodge module of weight $\dim Y$ on $Y$. The filtration $F_\bullet \omega_Y$ is given by ${\rm Gr}^F_{-\dim Y} \omega_Y = \omega_Y$ and the $\Q$-structure is $\Q_Y[\dim Y]$. For this reason, the Hodge module is denoted $\Q^H_Y[\dim Y]$.
\end{eg}

One of the most powerful aspects of Saito's theory is that $D^b ({\rm MHM}(-))$ admits a six functor formalism, in the sense that for any $f\colon X_1 \to X_2$ between complex algebraic varieties, there are exact functors
\[ f_*,f_! \colon D^b({\rm MHM}(X_1)) \to D^b({\rm MHM}(X_2)), \, f^*,f^! \colon D^b({\rm MHM}(X_2)) \to D^b({\rm MHM}(X_1))\]
lifting the corresponding functors of constructible complexes. Similarly, there are dual and tensor functors, but we will not need those in this note.

\begin{eg} \label{eg-HodgeFiltrationClosedEmbedding} Let $i \colon Y_1 \to Y_2$ be a closed embedding between smooth varieties. Let $(\cM,F)$ be a filtered right $\cD_{Y_1}$-module underlying a mixed Hodge module $M$ on $Y_1$. Then the right filtered $\cD_{Y_2}$-module underlying $i_* M$ is $i_+(\cM,F)$, whose $\cD$-module structure was described in \exampleref{eg-Kashiwara} above. The Hodge filtration is defined locally as follows: assume $y_1,\dots, y_c$ are part of a local system of coordinates which define $Y_1$ inside $Y_2$. Then
\[ i_+ \cM = \bigoplus_{\alpha \in \N^c} \cM \de_y^\alpha,\]
and we have
\[ F_p i_+ \cM = \bigoplus_{\alpha \in \N^c} F_{p-|\alpha|} \cM \de_y^{\alpha}.\]

In particular, if $p(\cM) = \min\{p \mid F_p \cM \neq 0\}$, we have
\[ F_{p(\cM)} i_+ \cM = F_{p(\cM)}\cM.\]
\end{eg}

If $a\colon X \to {\rm Spec}(\C)$ is the structure morphism, then we let $\Q^H_X = a^* \Q^H \in D^b({\rm MHM}(X))$, where $\Q^H$ is the trivial Hodge structure. For $Y$ smooth as above, the object $\Q^H_Y[\dim Y] \in D^b({\rm MHM}(Y))$ has a single non-vanishing cohomology module in degree $0$ which agrees with that in \exampleref{eg-smoothTrivial} above.

For $X$ singular, $\Q^H_X[\dim X]$ could have several cohomology modules, though only in non-positive degrees. Moreover, \cite{SaitoMHM}*{(4.5.9)} shows that if $X$ is equidimensional, then the weight graded piece ${\rm Gr}^W_{\dim X} \cH^0(\Q^H_X[\dim X])$ lies over the intersection complex perverse sheaf ${\rm IC}_X$, and so it is denoted ${\rm IC}_X^H$.

The category of mixed Hodge modules is also stable by nearby and vanishing cycles along any regular function $f\colon X \to \A^1$. Moreover, the morphisms
\[ {\rm Var} \colon \phi_{f,1}(M) \to \psi_{f,1}(M)(-1)\]
\[ {\rm can} \colon \psi_{f,1}(M) \to \phi_{f,1}(M)\]
are morphisms of mixed Hodge modules. 

As those morphisms represent $i^!(M)$ (resp. $i^*(M)[-1]$) \cite{SaitoMHM}*{Cor. 2.24}, where $i\colon V(f) \to X$ is the closed embedding, it is not hard to see that this implies that $M$ admits no non-zero sub-objects supported on $V(f)$ if and only if ${\rm Var}$ is injective, and it admits no non-zero quotient objects supported on $V(f)$ if and only if ${\rm can}$ is surjective (see \cite{SaitoMHP}*{Prop. 3.1.8} for a direct argument using properties of the $V$-filtration). 

As mentioned above, the $V$-filtration for $\cD$-modules is related to the operations of nearby and vanishing cycles. If $Y$ is a smooth variety with $f\in \cO_Y(Y)$ with graph embedding $\Gamma \colon Y \to Y \times \A^1_t$, then for $(\cM,F)$ underlying a mixed Hodge module $M$ on $Y$, we consider $\Gamma_+(\cM,F)$ on $Y\times \A^1_t$. As $\{t=0\}$ is a smooth hypersurface, the module $\Gamma_+(\cM)$ admits a $V$-filtration. The underlying $\cD_Y$-module of $\psi_{f,1}(M)$ (resp. $\phi_{f,1}(M)$) is ${\rm Gr}_V^1 \Gamma_*(\cM)$ (resp. ${\rm Gr}_V^0 \Gamma_*(\cM)$), with Hodge filtration given by
\[ F_\bullet \psi_{f,1}(M) = \frac{F_{\bullet-1} V^1 \cM}{F_{\bullet-1} V^{>1}\cM} \quad (\text{resp. } F_\bullet \phi_{f,1}(M) = \frac{ F_{\bullet} V^0 \cM}{F_{\bullet} V^{>0} \cM}).\]

The category of pure Hodge modules is stable under $f_* = f_!$ by the following foundational result of Saito:
\begin{thm} \cite{SaitoMHP}*{Thm. 1} Let $f\colon X \to Y$ be a projective morphism between complex algebraic varieties. Let $M$ be a pure Hodge module of weight $w$ on $X$. Then $\cH^i f_* M$ is a pure Hodge module of weight $w+i$ on $Y$.
\end{thm}

Moreover, Saito shows that every pure object $M$ admits a \emph{strict support decomposition} (in fact, this is true in more general categories of mixed sheaves \cite{SaitoMixedSheaves}*{Lem. 6.2}, hence is a natural requirement in any theory with a weight formalism). To define this, we say that a module $M'$ has strict support $Z$, where $Z$ is an irreducible closed subvariety of $X$, if ${\rm Supp}(M') = Z$ and if $M'$ admits no non-zero sub-object or quotient object supported on a proper closed subvariety of $Z$. As noted above, this implies that for any $f \in \cO_X(X)$ such that $f \vert_Z \neq 0$, the natural morphism
\[ {\rm Var} \colon \phi_{f,1}(M') \to \psi_{f,1}(M')(-1) \text{ is injective}\]
and the morphism
\[ {\rm can} \colon \psi_{f,1}(M') \to \phi_{f,1}(M')\text{ is surjective}.\]

If $f = t$ defines a smooth hypersurface, these morphisms are represented at the $\cD$-module level by multiplication by $t$ (resp. by $\de_t$). In general, one uses the graph embedding to reduce to the case of a smooth hypersurface.

\begin{eg} The standard example of a Hodge module with strict support $X$ (if $X$ is irreducible) is ${\rm IC}_X^H$, the intersection complex Hodge module.
\end{eg}

For any $M$ a pure Hodge module of weight $w$ on $X$, there exists a decomposition $M = \bigoplus_{Z\subseteq X} M_Z$ where $Z\subseteq X$ ranges over irreducible closed subsets of $X$ and $M_Z$ has strict support $Z$.

The last few statements we need from the theory of mixed Hodge modules concern commutativity of nearby and vanishing cycles across projective morphisms.

\begin{thm} \cite{SaitoMHM}*{Thm. 2.14} \label{thm-CyclesPushforward} Let $f\colon X \to Y$ be a projective morphism of complex algebraic varieties. Let $g \colon Y \to \A^1$ be a regular function. Then for any $M \in {\rm MHM}(X)$ and $i\in \Z$, there are canonical isomorphisms
\[ \psi_{g,1} \cH^i f_* M \cong \cH^i f_* \psi_{g\circ f,1}M\]
\[ \phi_{g,1} \cH^i f_* M \cong \cH^i f_* \phi_{g\circ f,1}M\]
of mixed Hodge modules on $Y$.
\end{thm}

For the underlying filtered $\cD$-modules, this is implied by a more general bistrictness result for the Hodge and $V$-filtrations under projective pushforward. That result also gives the following:
\begin{prop} \cite{BMS}*{Prop. 3.2} \label{prop-bistrictness} Let $Z\subseteq Y_2$ be an inclusion of smooth varieties and let $Y_1$ be another smooth variety. Let $p \colon Y_1\times Y_2 \to Y_2$ be the second projection and $(\cM,F)$ be a filtered right $\cD_{Y_1\times Y_2}$-module underlying a mixed Hodge module such that $p$ is projective on the support of $\cM$. If $V^\bullet \cM$ is the $V$-filtration along $Y_1 \times Z \subseteq Y_1\times Y_2$ and $p_0 = \min \{p \mid {\rm Gr}^F_p \cM \neq 0\}$, then
\[ F_{p_0} V^\alpha \cH^i p_+(\cM) = R^i p_*(F_{p_0}V^\alpha \cM),\]
where on the left $V^\alpha \cH^i p_+(\cM)$ is the $V$-filtration along $Z\subseteq Y_2$.
\end{prop}
\begin{proof} The reason the formula looks slightly different from that in \emph{loc. cit.} is that we are working with right $\cD$-modules, so rather than using the relative de Rham complex for the pushforward, we use the relative Spencer complex (see \cite{HTT}*{Pg. 44–46}). Then this statement is simply a side-changing of the standard one.
\end{proof}

\begin{eg} \label{eg-mainComputation} The following is the key computation in \cite{BS,BMS} and will be key to our argument as well. We state it in terms of right $\cD$-modules.

Let $\pi \colon Y_1 \to Y_2$ be a projective morphism between smooth varieties and let $f_1,\dots, f_r \in \cO_{Y_2}(Y_2)$. Let $g_i = \pi^*(f_i)$ for $1\leq i\leq r$. Then we have the Cartesian diagram
\[ \begin{tikzcd} Y_1 \ar[r] \ar[d,"\pi"] & Y_1 \times \A^r_t \ar[d,"\pi\times {\rm id}"] \\ Y_2 \ar[r] & Y_2 \times \A^r_t \end{tikzcd},\]
where the horizontal morphisms are the graph embeddings and the vertical morphisms are induced by $\pi$.

Let $(\cM,F)$ be a filtered right $\cD_{Y_1 \times \A^r_t}$-module underlying a mixed Hodge module on $Y_1 \times \A^r_t$. Our goal is to understand $F_{p_0} V^\lambda \cH^i (\pi \times {\rm id})_*(\cM)$ in terms of the $V$-filtration of $\cM$ along $t_1,\dots, t_r$.

To do this, factor $\pi = p \circ i \colon Y_1 \to Y_1\times Y_2 \to Y_2$ so that we have the commutative diagram
\[ \begin{tikzcd} Y_1\times \A^r_t \ar[r,"i \times {\rm id}"] \ar[dr,swap, "\pi\times {\rm id}"] & Y_1 \times Y_2\times \A^r_t \ar[d, "p \times {\rm id}"]\\ {} & Y_2 \times \A^r_t \end{tikzcd}.\]

Although the map $p\times {\rm id}$ need not be projective, it is certainly projective on the support of $M' = (i\times {\rm id})_*(M)$. Moreover, we can view $p \times {\rm id}$ as the second projection. By \propositionref{prop-bistrictness} applied to $M'$ (with underlying filtered right $\cD$-module $(\cM',F)$) and the zero section $Y_1 \times Y_2 \times \{0\}$, we conclude
\[ F_{p_0} V^\alpha \cH^i (p \times {\rm id})_* \cM' = R^i (p\times {\rm id})_*( F_{p_0} V^\alpha \cM' ),\]
and since the $V$-filtration (\exampleref{eg-InvarianceClosedEmbeddings}) and first non-zero piece of the Hodge filtration (\exampleref{eg-HodgeFiltrationClosedEmbedding}) essentially remain unchanged under closed embeddings, we conclude
\begin{equation} \label{eq-BiStrictComputation} F_{p_0} V^\alpha \cH^i p_*\cM = F_{p_0} V^\alpha \cH^i (p \times {\rm id})_* \cM' = R^i p_*(F_{p_0} V^\alpha \cM).\end{equation}
\end{eg}
 
We will use \theoremref{thm-CyclesPushforward} in the following application (see \cite{SaitoKollar}*{Prop. 2.6}, whose proof we give for convenience of the reader):
\begin{lem} \label{lem-lowestHodge} Let $f \colon X \to Y$ be a projective morphism between smooth complex varieties. Let $M$ be a pure Hodge module on $X$ with strict support $Z \subseteq X$. Let $p(M) = \min \{p \mid F_p \cM \neq 0\}$ where $(\cM,F)$ is the filtered right $\cD_X$-module underlying $M$. 

Then if $M'$ is a strict support summand of $\cH^i f_* M$ for some $i\in \Z$ such that ${\rm supp}(M') \neq f(Z)$, we have $p(M') > p(M)$.
\end{lem}
\begin{proof} Let $M'$ be a strict support summand of $\cH^i f_* M$ with ${\rm supp}(M') \neq f(Z)$. Let $g \in \cO_Y$ be a locally defined function which vanishes on ${\rm supp}(M')$ and not on $f(Z)$. In particular, $Z$ is not contained in the zero locus of $g\circ f$.

As ${\rm supp}(M') \subseteq g^{-1}(0)$, we see that $\phi_{g,1}(M') \cong M'$. Thus, we see that $M'$ is a direct summand of $\phi_{g,1} \cH^i f_* M$. By the compatibility of vanishing cycles with projective pushforwards, we see that $M'$ is a direct summand of $\cH^i f_* \phi_{g\circ f,1}(M)$.

Finally, as $M$ has strict support not contained in $(g\circ f)^{-1}(0)$, we see that the natural morphism
\[ {\rm can} \colon \psi_{g\circ f,1}(M) \to \phi_{g\circ f,1}(M)\]
is surjective. For the underlying filtered $\cD$-modules, let $\Gamma$ be the graph embedding along $g\circ f$. Then this says that
\[  \frac{F_{p(M)} V^0 \Gamma_+ \cM}{F_{p(M)} V^{>0}\Gamma_+ \cM}  = F_{p(M)} \phi_{t,1}(\Gamma_+ \cM) = F_{p(M)} \psi_{t,1}(\Gamma_+ \cM) \cdot \de_t = \frac{F_{p(M)-1} V^1\Gamma_+ \cM}{F_{p(M)-1} V^{>1}\Gamma_+ \cM}   =  0,\]
where the vanishing of the right hand side follows from definition of $p(M)$.

Thus, $p(\phi_{g\circ f,1}(M)) > p(M)$, and so we conclude by
\[ p(M') \geq p(\cH^i f_* \phi_{g\circ f,1}(M)) \geq p(\phi_{g\circ f,1}(M)) > p(M).\]
\end{proof}

We conclude with some lemmas from the theory of filtered $\cD$-modules which will be useful in the proofs of the main results. 

Let $Z\subseteq Y$ be smooth varieties. For a filtered $\cD_Y$-module $\cM$ let $p(\cM) = \min\{p \mid F_p \cM \neq 0\}$. Define $G^\bullet_Z \cM = (F_{p(\cM)}\cM)\cdot V_Z^\bullet \cD_Y$, which is not necessarily an exhaustive filtration. We let $b_{Z,\cM}(s)$ denote the minimal polynomial of the action of $\theta$ on ${\rm Gr}_{G_Z}^0 \cM$, where $\theta$ is any derivation in $V^0_Z\cD_Y$ which acts as the identity of $\cI_Z/\cI_Z^2$. The main example of $\theta$ for us will be the Euler operator $\sum t_i \de_{t_i}$ when $Z \subseteq Y$ is defined in a local partial system of coordinates by $t_1,\dots, t_r$.

It is not hard to see that
\[ ({\rm Gr}_{G_Z}^j \cM) b_{Z,\cM}(\theta+j) = 0 \text{ for all } j \in \Z.\]

\begin{eg}\label{eg-ZeroEmbedding} Let $i \colon Z \to Y$ be a closed embedding of smooth varieties. Let $(\cM,F) = i_*(\cM',F)$ for some filtered $\cD_Z$-module $(\cM',F)$. Then $b_{Z,\cM}(s) = s$.

Indeed, assume $Z$ is defined inside $Y$ by a partial system of coordinates $t_1,\dots, t_r$ on $Y$, with corresponding vector fields $\de_{t_1},\dots, \de_{t_r}$. Then we can write
\[ \cM = i_+ \cM = \bigoplus_{\alpha \in \N^r} \cM' \delta \de_t^\alpha\]
where $m \delta t_i = 0$ for $1\leq i\leq r$ and any $m\in \cM$. Also,
\[ F_p \cM = \bigoplus_{\alpha \in \N^r} F_{p-|\alpha|} \cM'\delta \de_t^\alpha.\]

In particular, we see that $p(\cM) = p(\cM')$, which we denote by $p$, and moreover that $F_p \cM = (F_p \cM')\delta$. Thus,
\[ G^k_Z \cM = \begin{cases} (F_p \cM' \cdot \cD_Z)\delta & k = 0 \\ 0 & k =1 \end{cases},\]
and we conclude that $({\rm Gr}_{G_Z}^0 \cM) \theta = (G_Z^0 \cM)\theta = 0$ as $\theta = \sum t_i \de_{t_i}$. This proves $b_{Z,\cM}(s) = s$.
\end{eg}

\begin{lem} \label{lem-sameSubvariety} Let $Z\subseteq Y_1 \subseteq Y_2$ be smooth varieties and denote $i\colon Y_1 \to Y_2$ the closed embedding. Let $\cM$ be a $\cD_{Y_1}$-module. Then
\[F_{p(\cM)} V_Z^\bullet i_+ \cM = F_{p(\cM)} V_Z^\bullet \cM\]
and
\[ b_{Z,\cM}(s) = b_{Z,i_+ \cM}(s).\]
\end{lem}
\begin{proof} We take local coordinates on $Y_2$ such that $Y_1$ is defined by $y_1,\dots, y_c$ and $Z$ is defined by $y_1,\dots, y_c, t_1,\dots, t_r$ (so that $Z$ is defined inside $Y_1$ by the restriction of $t_1,\dots, t_r$). 

In these local coordinates, we have
\[ i_+ \cM = \bigoplus_{\alpha \in \N^c} \cM \delta\de_y^\alpha\]
and
\[ F_p i_+\cM = \bigoplus_{\alpha \in \N^c} F_{p-|\alpha|} \cM \delta\de_y^\alpha,\]
where $\delta$ is a formal symbol such that $m\delta y_i = 0$ for all $1\leq i\leq c$.

The $V$-filtration $V^\bullet_Z i_+\cM$ concerns the Euler operator $\theta_y + \theta_t$, and the $V$-filtration $V^\bullet_Z \cM$ only concerns $\theta_t$. We have
\begin{equation} \label{eq-EulerOperator} m\delta \de_y^\alpha (\theta_y+\theta_t) = (m\delta(\theta_y + |\alpha| + \theta_t)) \de_y^\alpha= (m (\theta_t + |\alpha|)\delta) \de_y^\alpha,\end{equation}
using that $m\delta y_i = 0$ for $1\leq i \leq c$, and so it is not hard to check that
\[ V_Z^\bullet i_+ \cM = \bigoplus_{\alpha \in \N^c} V_Z^{\bullet+|\alpha|}\cM \delta \de_y^\alpha.\]

This, along with the formula for the Hodge filtration, gives the first claimed equality.

For the $b$-function claim, note that
\[ (F_{p(\cM)}\cM \delta)\cdot V_Z^i \cD_{Y_2} = \sum_{\alpha \in \N^c} (F_{p(\cM)}\cM \cdot V_Z^{i+|\alpha|}\cD_{Y_1}\delta) \de_y^\alpha,\]
for all $i\geq 0$, or in other words,
\[ G_Z^i i_+ \cM = \sum_{\alpha \in \N^c} ( G_Z^{i+|\alpha|} \cM \delta)\de_y^\alpha,\]
and so the claim follows once again by the relation \ref{eq-EulerOperator} and the fact that 
\[ ({\rm Gr}_{G_Z}^j \cM)\cdot b_{Z,\cM}(\theta_t+j) =0.\]
\end{proof}

\begin{lem} \label{lem-sameFunctions} Let
\[ \begin{tikzcd} Z\cap Y_1 \ar[r] \ar[d] & Z\ar[d] \\ Y_1 \ar[r] & Y_2\end{tikzcd}\]
be a Cartesian diagram of smooth closed subvarieties in a smooth variety $Y_2$ and let $i\colon Y_1 \to Y_2$ be the closed embedding. Let $\cM$ be a $\cD_{Y_1}$-module with pushforward $i_+ \cM$. Then
\[ F_{p(\cM)} V_Z^\bullet i_+ \cM =F_{p(\cM)} V_{Z\cap Y_1}^{\bullet}\cM \]
and
\[ b_{Z\cap Y_1,\cM}(s) = b_{Z,i_+\cM}(s).\]
\end{lem}
\begin{proof} This is also a local computation, so we take local coordinates on $Y_2$ such that $Y_1$ is defined by $y_1,\dots, y_c$ inside $Y_2$ and $Z$ is defined by $t_1,\dots, t_r$, so that $Z\cap Y_1$ is defined by $y_1,\dots, y_c, t_1,\dots, t_r$.

In these coordinates, we have
\[ i_+ \cM = \bigoplus_{\alpha \in \N^c} \cM \delta \de_y^\alpha\]
where $m\delta y_i = 0$ for all $1\leq i\leq c$ and the Hodge filtration is given by
\[ F_p i_+ \cM = \bigoplus_{\alpha \in \N^c} (F_{p-|\alpha|} \cM \delta)\de_y^\alpha.\]

It is easy to see, then, that
\[ V_Z^\bullet i_+\cM = \bigoplus_{\alpha \in \N^c} (V_{Z\cap Y_1}^\bullet \cM \delta)\de_y^\alpha\]
because
\[ (m\delta)\de_y^\alpha \theta_t = m\theta_t \delta \de_y^\alpha,\]
 so we get the first claim.

For the $b$-function claim, note that
\[ (F_{p(\cM)}\cM \delta)\cdot V_Z^i \cD_{Y_2} = \sum_{\alpha \in \N^c} (F_{p(\cM)}\cM \cdot V_{Z\cap Y_1}^i\cD_{Y_1}\delta) \de_y^\alpha,\]
for all $i\geq 0$, or in other words,
\[ G_Z^i i_+ \cM = \sum_{\alpha \in \N^c} ( G_Z^{i} \cM \delta)\de_y^\alpha,\]
and so the claim follows due to the fact that $\theta_t$ commutes with $\de_y^\alpha$.
\end{proof}

\section{Proofs} \label{sec-proofs}
In this section, we prove the main theorem and its corollaries. As mentioned above, the proof of the main theorem is essentially the same as that in \cite{BS,BMS}, with the caveat that we use \lemmaref{lem-lowestHodge} instead of a restriction on the $V$-filtration pieces in \emph{loc. cit.}. 

\begin{proof}[Proof of \theoremref{thm-main}] As in the statement, let $f_1,\dots, f_r \in \cO_X(X)$ and denote by the same symbol their extensions to $Y$. Let $\Gamma \colon Y \to Y \times \A^r_t$ be the graph embedding along $f_1,\dots, f_r$.

Let $\pi \colon \widetilde{X} \to X$ denote a log resolution of $(X,\fra)$, so that $\widetilde{X}$ is smooth and $\fra \cdot \cO_{\widetilde{X}} =  \cO_{\widetilde{X}}(-F)$ where $F$ is a divisor on $\widetilde{X}$ with SNC support. We let $g_i = \pi^*(f_i)$ for $1\leq i \leq r$ and let $\Gamma_g \colon \widetilde{X} \to \widetilde{X}\times \A^r_t$ be the graph embedding.

As in \cite{BMS}*{Pf. of Theorem 1}, we consider the diagram
\[ \begin{tikzcd} & \widetilde{X}'\ar[d,"i'"] \\ \widetilde{X} \ar[r,"\Gamma_g"] \ar[d]  \ar[ru,"i_{g'}"] & \widetilde{X}\times \A^r_t \ar[d]\\ Y \ar[r] & Y \times \A^r_t \end{tikzcd}.\]

The space $\widetilde{X}'$ is the total space of the line bundle $\cO_{\widetilde{X}}(F)$ on $\widetilde{X}$, the morphism $i' \colon \widetilde{X}' \to \widetilde{X}\times \A^r_t$ is an inclusion of vector bundles and the map $i_{g'} \colon \widetilde{X} \to \widetilde{X}'$ is induced by the section $\cO_{\widetilde{X}} \to \cO_{\widetilde{X}}(F)$ given by the containment $\cO_{\widetilde{X}}(-F) \to \cO_{\widetilde{X}}$. In particular, if $g'$ is a local defining equation for $F$, which trivializes $\cO_{\widetilde{X}}(F)$, then the map $i_{g'} \colon \widetilde{X} \to \widetilde{X}'$ looks like the graph embedding for $g'$, which explains the notation.

Importantly, $\widetilde{X}$ also sits inside $\widetilde{X}'$ and $\widetilde{X}\times \A^r_t$ as the zero section in a compatible way. By \lemmaref{lem-sameSubvariety}, the $V$-filtration along this zero section is essentially unchanged when applying the closed embedding $i'_+$. By \cite{BS}*{Thm. 0.1}, we see that
\[ F_{-\dim X} V^\alpha  i_{g',+} \omega_{\widetilde{X}} = \omega_{\widetilde{X}}(-\lfloor (\alpha -\varepsilon)F\rfloor),\]
(which follows by an explicit computation for normal crossing type Hodge modules \cite{SaitoMHM}*{(3.5.1)}) and hence by applying $i'_+$, we conclude
\[ F_{-\dim X} V^\alpha \Gamma_{g,+} \omega_{\widetilde{X}} = \omega_{\widetilde{X}}(-\lfloor (\alpha - \varepsilon)F\rfloor).\]

We conclude by the identification \ref{eq-BiStrictComputation} that
\[ F_{-\dim X} V^\alpha \cH^0 (\pi \times {\rm id})_+( \omega_{\widetilde{X}}) = \pi_*( \omega_{\widetilde{X}}(-\lfloor (\alpha - \varepsilon)F\rfloor)).\]

Finally, the canonical map $\Q_X^H[\dim X ] \to \pi_* \Q^H_{\widetilde{X}}[\dim X]$ gives a canonical map
\[ {\rm IC}_X^H \to \cH^0 \pi_* \Q^H_{\widetilde{X}}[\dim X]\]
of pure Hodge modules of weight $\dim X$. Importantly, if we apply $F_{-\dim X}$, we get an isomorphism by \lemmaref{lem-lowestHodge}
\[ F_{-\dim X} {\rm IC}_X^H \cong F_{-\dim X} \cH^0\pi_* \Q^H_{\widetilde{X}}[\dim X],\]
and hence if we apply $V^\alpha \Gamma_+$ to this equality, we conclude the desired equality.
\end{proof}

\begin{rmk} \label{rmk-SameProof} In the original argument of \cite{BMS}, where the variety $X$ is smooth, a strong log resolution was used. Thus, the map $\pi \colon \widetilde{X} \to X$ in that setting was an isomorphism over the complement of $\{f=0\}$ in $X$. The argument in \emph{loc. cit.} differed from ours in this way, because rather than invoking \lemmaref{lem-lowestHodge}, they were able to conclude using the fact that $V^\alpha \Gamma_+ \cM = 0$ for $\alpha > 0$, when $\cM$ is supported on $\{f = 0\}$ (here $\Gamma_+$ is the graph embedding along $f$), which follows by Example \ref{eg-Kashiwara}. This would apply here if $X_{\rm sing} \subseteq V(f)$, and is the essential difference in our arguments.
\end{rmk}

\begin{proof}[Proof of \corollaryref{cor-Skoda}] We have by \cite{CD}*{Thm. 1.1} the exactness of the Koszul-like complexes
\begin{equation} \label{eq-complexA} F_{-\dim X} V^\lambda \cM \xrightarrow[]{t} \bigoplus_{i=1}^r F_{-\dim X} V^{\lambda+1} \cM \xrightarrow[]{t}  \dots \xrightarrow[]{t}  F_{-\dim X} V^{\lambda+r} \cM\end{equation}
\begin{equation} \label{eq-complexB} F_{-\dim X} {\rm Gr}_V^\lambda \cM \xrightarrow[]{t} \bigoplus_{i=1}^r F_{-\dim X} {\rm Gr}_V^{\lambda+1} \cM\xrightarrow[]{t}\dots \xrightarrow[]{t}  F_{-\dim X} {\rm Gr}_V^{\lambda+r} \cM\end{equation}
for any $\lambda > 0$ (we recall that we are indexing in a non-standard way for right $\cD$-modules). 

The first claim follows by the main theorem, which implies that $\lambda$ is a jumping number of $\cJ(\omega_X,\fra^{\bullet})$ if and only if $F_{-\dim X}{\rm Gr}_V^{\lambda} \Gamma_+ {\rm IC}_X \neq 0$ and by injectivity on the left in the complex \ref{eq-complexB}, and the second claim follows by the main theorem and surjectivity on the right in the complex \ref{eq-complexA}.

The third claim is also local. Let $\lceil {\rm lct}(\omega_X,\fra)\rceil = j$. In particular, ${\rm lct}(\omega_X,\fra) > j-1$, so $\cJ(\omega_X,\fra^\mu) = \omega_X^{\rm GR}$ for all $\mu \leq j-1 + \varepsilon'$ for any $0 < \varepsilon' \ll 1$. Let $0 < \varepsilon_1 < \varepsilon_2 \ll 1$ and consider the exact complex 
\[ F_{-\dim X} V^{\varepsilon_2} \Gamma_+ {\rm IC}_X^H \xrightarrow[]{t} \bigoplus_{i=1}^r F_{-\dim X} V^{\varepsilon_2+1} \Gamma_+ {\rm IC}_X^H \xrightarrow[]{t}  \dots \xrightarrow[]{t}  F_{-\dim X} V^{\varepsilon+r}\Gamma_+ {\rm IC}_X^H,\]
which by the definition of the $\cD$-module action on $\Gamma_+ {\rm IC}_X^H$ and \theoremref{thm-main} can be identified with the Koszul-like complex
\[ F_{-\dim X} \cJ(\omega_X,\fra^{\varepsilon_2-\varepsilon_1}) \xrightarrow[]{f} \bigoplus_{i=1}^r \cJ(\omega_X,\fra^{\varepsilon_2-\varepsilon_1+1}) \xrightarrow[]{f}  \dots \xrightarrow[]{f}  \cJ(\omega_X,\fra^{\varepsilon_2-\varepsilon_1+r}),\]
by making $\varepsilon_2 - \varepsilon_1 \leq \varepsilon'$, the first $j+1$ terms of this complex are (by the assumption on ${\rm lct}(\omega_X,\fra)$)
\[ \omega_X^{\rm GR} \xrightarrow[]{f_i} \bigoplus_{i=1}^r \omega_X^{\rm GR} \to \dots \to \bigoplus_{|I| = j-1} \omega_X^{\rm GR} \to \bigoplus_{|I| = j} \cJ(\omega_X,\fra^{\varepsilon_2 -\varepsilon_1+j}),\]
which is exact. Thus, if we localize at any $\frm$ a maximal ideal sheaf containing $\fra$ (equivalently, take the stalk at any closed point $x \in V(\fra) \subseteq X$), then we conclude by the Koszul complex characterization of depth.
\end{proof}

We next prove \corollaryref{cor-DualDuBois}. We remind the reader of the definition of the invariant ${\rm HRH}(X)$ for an equidimenisonal variety $X$. For general properties, see \cite{DOR}.

For any complex variety $X'$, write $\mathbf D_{X'} = \mathbf D(\Q_{X'}^H[\dim X'])(-\dim X')$. Using the fact that ${\rm Gr}^F_{-p} {\rm DR}(\Q_X^H[\dim X]) = \underline{\Omega}_X^p[\dim X-p]$, we see that
\[ {\rm Gr}^F_{-p} {\rm DR}(\mathbf D_X^H) = \mathbb D(\underline{\Omega}_X^{\dim X-p})[-p].\]

\begin{defi} \label{defi-HRH} For an equidimensional variety $X$, we have the natural morphism
\[ \psi_X \colon \Q_X^H[\dim X] \to \mathbf D_X^H\]
and we define
\[ {\rm HRH}(X) = \sup\{k \mid {\rm Gr}^F_{-p}(\psi_X) \text{ is a quasi-isomorphism for all } p \leq k\}.\]

Equivalently, if $\gamma_X^{\vee} \colon {\rm IC}_X^H \to \mathbf D_X^H$ is the Tate-twisted dual of the map $\gamma_X \colon \Q_X^H[\dim X] \to {\rm IC}_X^H$, then 
\[ {\rm HRH}(X) = \sup\{k\mid {\rm Gr}^F_{p-\dim X}(\gamma_X^{\vee}) \text{ is a quasi-isomorphism for all } p \leq k\}.\]
\end{defi}

We see by the second characterization that ${\rm HRH}(X) \geq 0$ is equivalent to the claim that the map $\omega_{X}^{\rm GR} \to \mathbb D(\underline{\Omega}_X^0)[-\dim X]$ is a quasi-isomorphism. In particular, $\mathbb D(\underline{\Omega}_X^0)[-\dim X]$ is a coherent sheaf, hence $\underline{\Omega}_X^0$ is a maximal Cohen-Macaulay object in $D^b_{\rm coh}(\cO_X)$. In fact, the second characterization shows that the condition ${\rm HRH}(X) \geq k$ is equivalent to $X$ satisfying $(*)_k$ in \cite{ParkPopa}.

We also define the related invariant $c(X)$ defined in \cite{CDO2}, again assuming $X$ is equidimensional. This is a Hodge theoretic weakening of the condition that $\Q_X[\dim X]$ is a perverse sheaf.
\begin{defi} Let $X$ be equidimensional. Then we let
\[ c(X) = \sup \{k \mid {\rm Gr}^F_{-p} {\rm DR}(\tau^{<0} \Q_X^H[\dim X]) = 0\text{ for all } p \leq k\},\]
or, equivalently, by duality,
\[ c(X) = \sup \{k \mid {\rm Gr}^F_{p-\dim Z} {\rm DR}(\cH^0 \mathbf D_X^H) \cong {\rm Gr}^F_{p-\dim Z} {\rm DR}(\mathbf D_X^H) \text{ for all } p\leq k\}.\]
\end{defi}

These invariants are related by the inequality $c(X) \geq {\rm HRH}(X)$. We first prove a lemma concerning a restriction property, of independent interest. It is analogous to the property that if $Z$ is defined by a regular sequence inside a Cohen-Macaulay variety $X$, then $Z$ is Cohen-Macaulay. Indeed, the condition $c(X) \geq k$ is equivalent to ${\rm depth}\, \underline{\Omega}_X^p \geq \dim X - p$ for all $p\leq k$ by \cite{CDO2}*{Thm. A}.

\begin{lem} \label{lem-adjunction} Assume $X$ is irreducible and $f_1,\dots, f_r \in \cO_X(X)$ form a regular sequence defining $Z\subseteq X$ reduced of pure codimension $r$. Then $c(X) \leq c(Z)$.
\end{lem}
\begin{proof} Assume $c(X) \geq k$. We will show $c(Z) \geq k$.

We have the exact triangle
\[ \cH^0 \mathbf D_X^H \to \mathbf D_X^H \to \tau^{>0} \mathbf D_X^H \xrightarrow[]{+1}.\]

Let $i\colon Z \to X$ be the closed embedding. Applying $i^!(-)$ to this triangle and Tate twisting, we get the triangle
\begin{equation} \label{eq-exactTriangle} i^! \cH^0 \mathbf D_X^H (r)[r] \to \mathbf D_Z^H \to i^! \tau^{>0} \mathbf D_X^H (r)[r] \xrightarrow[]{+1}.\end{equation}

As $f_1,\dots ,f_r$ form a regular sequence, we see that $\cH^j i^!\cH^0 \mathbf D_X^H(r) = 0$ for all $j > r$. Indeed, the claim is local and only depends on underlying $\cD$-modules, so we can assume $i_{X} \colon X\hookrightarrow Y$ is a closed embedding into smooth $Y$ such that $f_i = F_i\vert_{X}$ for some $F_i \in \cO_Y(Y)$. Then, if $\Gamma \colon Y \to Y\times \A^r_t$ is the graph embedding along $F_1,\dots, F_r$ (and $\gamma \colon X \to X\times \A^r_t$ is the graph embedding along $f_1,\dots, f_r$), we have the Cartesian diagram
\[ \begin{tikzcd} X \ar[r,"\gamma"] \ar[d] & X\times \A^r_t \ar[d] \\ Y \ar[r,"\Gamma"] & Y \times \A^r_t \end{tikzcd}.\]

Then if $\sigma_X\colon X \to X\times \A^r_t ,\sigma_Y\colon Y \to Y \times \A^r_t$ are the inclusions of the zero sections, we have the following isomorphisms by applying Base-Change (\cite{SaitoMHM}*{(4.4.3)}) twice:
\[ i_{X,*} i_* i^! \cH^0 \mathbf D_X^H \cong i_{X,*} \sigma_X^! \gamma_*\cH^0 \mathbf D_X^H \cong \sigma_Y^! \Gamma_* i_{X,*} \cH^0 \mathbf D_X^H.\]

Let $\cM'$ be the underlying right $\cD_Y$-module of $i_{X,*}\cH^0 \mathbf D_X^H$. By \cite{CD}*{Thm. 1.2}, the right-most complex is computed using the $V$-filtration on $\Gamma_* \cM'$ to form the Koszul-like complex similarly to the complex \ref{eq-complexA} (taking $\lambda = 0$, and the union over all Hodge filtration pieces). In particular, it only has possibly non-zero cohomology in degrees $\{0,\dots, r\}$, proving the vanishing.

Thus, we have a quasi-isomorphism
\begin{equation} \label{eq-qis} \tau^{>0} \mathbf D_Z^H \cong \tau^{>r} i^!\tau^{>0} \mathbf D_X^H(r).\end{equation} 

By definition of $i^!$ for mixed Hodge modules (or applying \cite{CD}*{Thm. 1.2} again), we conclude from $c(X) \geq k$ that ${\rm Gr}^F_{p-\dim X} {\rm DR}(\tau^{>r} i^!\tau^{>0} \mathbf D_X^H) = 0$ for all $p\leq k$ in this case. Thus, by the quasi-isomorphism \ref{eq-qis}, we conclude ${\rm Gr}^F_{p-\dim Z}{\rm DR}( \tau^{>0} \mathbf D_Z^H) = 0$ for all $p\leq k$, hence $c(Z) \geq k$.
\end{proof}

We can now prove the corollary.

\begin{proof}[Proof of \corollaryref{cor-DualDuBois}]  Assume ${\rm HRH}(X) \geq 0$. The inequality $0 \leq {\rm HRH}(X) \leq c(X)$ implies, by \lemmaref{lem-adjunction}, that $c(Z) \geq 0$.

This condition is equivalent to the quasi-isomorphism
\[ {\rm Gr}^F_{-\dim Z} {\rm DR}(\cH^0 \mathbf D_Z^H) \cong {\rm Gr}^F_{-\dim Z} {\rm DR}(\mathbf D_Z^H).\]

We have an exact triangle
\[ {\rm IC}_X^H \to \mathbf D_X^H \to \widetilde{K}_X^\bullet\xrightarrow[]{+1},\]
where $\widetilde{K}_X^\bullet$ is the dual of the \emph{RHM defect object} of \cite{ParkPopa} (up to Tate twist). We see that ${\rm HRH}(X) \geq 0$ is equivalent to ${\rm Gr}^F_{-\dim X} {\rm DR}(\widetilde{K}_X^{\bullet}) = 0$.

If we apply $i^!$ to this triangle, shift and Tate twist, we get
\[ i^! {\rm IC}_X^H[r](r) \to \mathbf D_Z^H \to i^! \widetilde{K}_X^\bullet[r](r) \xrightarrow[]{+1}.\]

Thus, the condition ${\rm HRH}(X) \geq 0$ implies 
\[ {\rm Gr}^F_{-\dim X}{\rm DR}(\cH^r i^! {\rm IC}_X^H) \cong {\rm Gr}^F_{-\dim Z} {\rm DR} \cH^0 \mathbf D_Z^H,\]
and so combining these quasi-isomorphisms, we get
\[ {\rm Gr}^F_{-\dim X} {\rm DR}(\cH^r i^!{\rm IC}_X^H) \cong {\rm Gr}^F_{-\dim Z} \mathbf D_Z^H = \mathbb D(\underline{\Omega}_Z^0)[-\dim Z].\]
\theoremref{thm-main} allows us to compute the left hand side in terms of the multiplier modules. Indeed, we do so locally, so take once again embedding $X\hookrightarrow Y$ with the notation as above. We use the same Base-Change argument as above to reduce to computing $\sigma_Y^! \Gamma_* i_{X,*}{\rm IC}_X^H$, where $\sigma_Y \colon Y \to Y \times \A^r_t$ is the inclusion of the zero section.

If $(\cM,F)$ is the right filtered $\cD$-module underlying $i_{X,*}{\rm IC}_X^H$, then we can compute 
\[{\rm Gr}^F_{-\dim X} {\rm DR} \sigma_Y^! \Gamma_* i_{X,*} {\rm IC}_X^H = F_{-\dim X} \sigma_Y^! \Gamma_* i_{X,*}{\rm IC}_X^H\] in terms of the $V$-filtration (along $t_1,\dots, t_r$) of $\cM$ by \cite{CD}*{Thm. 1.2}. We get
\[ F_{-\dim X} \cH^r \sigma_Y^! \Gamma_* i_*\cM = \frac{F_{-\dim X} V^r \Gamma_* \cM}{\fra \cdot F_{-\dim X} V^{r-1} \Gamma_*\cM},\]
which, with \theoremref{thm-main}, proves the first claim. 

For the second claim, if ${\rm lct}(\omega_X,\fra) = r$, then
\[ \cJ(\omega_X,\fra^{r-\varepsilon}) = \cJ(\omega_X,\fra^{r-1-\varepsilon}) = \omega_X^{\rm GR} \text{ for all } 0 < \varepsilon \ll 1,\]
and in this case, ${\rm depth}(\omega_X^{\rm GR},\fra) =r$ by \corollaryref{cor-Skoda}, so we have
\[ \frac{\omega_X^{\rm GR}}{\fra \cdot \omega_X^{\rm GR}} = \cO_Z \otimes_{\cO_X}^L \omega_X^{\rm GR}.\]

Thus, we get
\[ \mathbb D(\underline{\Omega}_Z^0)[-\dim Z] \cong  \cO_Z\otimes_{\cO_X}^L \omega_X^{\rm GR} .\]

We can rewrite $\cO_Z \otimes_{\cO_X}^L (-)$ as $Li^*(-)$ where $i\colon Z \to X$ is the closed embedding. Then $Li^* \circ \mathbb D_X(-) = \mathbb D_Z\circ i^!(-)$, where $i^!(-)$ is the exceptional pullback, in this case (since $i \colon Z \to X$ is defined by a regular sequence):
\[ i^!(-) = Li^*(-) [-r],\]
and so we can rewrite the isomorphism as
\[ \mathbb D_Z(\underline{\Omega}_Z^0)[-\dim Z] = i^!(\omega_X^{\rm GR})[r].\]

If we apply duality to both sides, we get
\[ \underline{\Omega}_Z^0 = L i^* (\mathbb D_X(\omega_X^{\rm GR})[\dim X]).\]

So we see, for any resolution of singularities $\pi \colon Y \to X$, that we have
\[ \underline{\Omega}_Z^0 = L i^* R\pi_*(\cO_Y).\]
\end{proof}

For \corollaryref{cor-Ajit}, we recall the set-up of the deformation to the normal cone. Again, we will make use of an embedding $X \hookrightarrow Y$ and assume $f_1,\dots, f_r \in \cO_X(X)$ are restricted from $Y$. We let $\Gamma \colon Y \to Y \times \A^r_t = T$ denote the graph embedding along $f_1,\dots, f_r$. 

As in the introduction, we define $\widetilde{X}$ and $\widetilde{Y}$ (resp. $\widetilde{T}$) to be the deformation to the normal cone spaces along $(f_1,\dots, f_r)$ (resp. $(t_1,\dots, t_r)$). We have the commutative diagram with Cartesian squares
\[ \begin{tikzcd} \widetilde{X} \ar[r] \ar[d] & \widetilde{Y} \ar[r] \ar[d] & \widetilde{T} \ar[d] \\ X\ar[r] & Y\ar[r] & T \end{tikzcd}.\]

Note that $\widetilde{T} = Y \times \A^r_z \times \A^1_u$ with morphism $\psi$ down to $T$ given at the ring level by $t_i \mapsto z_i u_i$. This morphism is clearly not smooth, however, when we restrict to $\{u\neq 0\} \cong T \times \mathbf G_{m,u}$, the morphism is the projection down to $T$. 

We consider two mixed Hodge modules on $\widetilde{T}$: the first is simply ${\rm IC}_{\widetilde{X}}$. The latter is constructed as follows: we consider the graph embedding $\Gamma$, which we apply to ${\rm IC}_X$ to get a pure Hodge module on $T$. We pull this back to $\{u\neq 0\}$ by the projection, giving
\[ \Gamma_* {\rm IC}_X \boxtimes \Q^H_{\mathbf G_{m,u}}[1],\]
and then we apply $j_*$ to this to get the object $\widetilde{\rm IC}_{X}$ on $\widetilde{T}$.

Note that the restriction to $\{u\neq 0\}$ of $\widetilde{\rm IC}_X$ and that of ${\rm IC}_{\widetilde{X}}$ are isomorphic (both to ${\rm IC}_{X\times \mathbf G_{m,u}}$). Hence, we have a canonical isomorphism
\[ V^\lambda \widetilde{\rm IC}_X \cong V^\lambda {\rm IC}_{\widetilde{X}} \text{ for all } \lambda > 0,\]
and similarly, an isomorphism
 \begin{equation} \label{eq-IsoDefNormalCone}  F_{p} V^\lambda \widetilde{\rm IC}_X \cong F_{p} V^\lambda {\rm IC}_{\widetilde{X}} \text{ for all } \lambda > 0 \text{ and } p\in \Z.\end{equation}

The point is that we can actually compute the $V$-filtration on $\widetilde{\rm IC}_{X}$. Indeed, by definition of the box product and our choice of coordinate $u$ on the $\A^1_u$ term in $\widetilde{T}$, we have that the underlying $\cO$-module is given by
\[ \bigoplus_{k\in \Z} \Gamma_* \cM u^k du,\]
but to determine the $\cD$-module action, we must use the change of coordinates formula. By choosing local coordinates on $Y$, say $y_1,\dots, y_n$, we get on $\{u\neq 0\}$ the coordinates $y_1,\dots, y_n, t_1,\dots, t_r,u$. Similarly, we have the coordinates $y_1,\dots, y_n, z_1,\dots, z_r, u$ where $z_i = \frac{t_i}{u}$ for all $i$. We use a $\widetilde{(-)}$ to denote $y_1,\dots, y_n, u$ when viewed in the first system of coordinates. The change of coordinates formula says
\[ \de_{t_i} = \de_{t_i}(z_i) \de_{z_i} = \frac{1}{u} \de_{z_i},\]
\[ \de_{\widetilde{u}} = \de_u(u) \de_u + \sum_{i=1}^r \de_u(z_i) \de_{z_i} = \de_u -\sum_{i=1}^r \frac{t_i}{u^2} \de_{z_i}\]
\[ = \de_u - \frac{1}{u} \sum_{i=1}^r t_i \de_{t_i}.\]

In particular, we see from the last equality that $\theta_u = \theta_{\widetilde{u}} + \theta_t$. Hence, for a local section $m\otimes \alpha$ of $\Gamma_* \cM$, we see that
\[ (m \otimes \alpha u^k du) (\theta_u) = (m\otimes \alpha u^k du)(\theta_{\widetilde{u}} + \theta_t) = (sm \otimes \alpha u^k du) + (m \otimes \alpha u^k du)(\theta_{\widetilde{u}})\]
\[ = (s - (k+1))m \otimes \alpha u^k du.\]

From this computation, it is not hard to check the following formula:
\[ V^{\lambda} \widetilde{\rm IC}_{X} = \bigoplus_{k \in \Z} V^{\lambda-k-1} \Gamma_* \cM u^k du,\]
and we similarly have, by a standard argument (see \cite{CD}*{Lem. 2.4})
\[ F_{-\dim X-1} V^{\lambda} \widetilde{\rm IC}_{X} = \bigoplus_{k \in \Z} F_{-\dim X} V^{\lambda-k-1} \Gamma_* \cM u^k du.\]

\begin{proof} [Proof of \corollaryref{cor-Ajit}]By the main theorem (used for the leftmost and rightmost equalities) and the isomorphism \ref{eq-IsoDefNormalCone}, we get the formula
\[ \cJ(\widetilde{X},u^{\lambda-\varepsilon}) = F_{-\dim \widetilde{X}} V^\lambda {\rm IC}_{\widetilde{X}} = F_{-\dim X -1}V^\lambda \widetilde{\rm IC}_{X} =  \bigoplus_{k \in \Z} F_{-\dim X} V^{\lambda-k-1} \Gamma_* \cM u^k du\]
\[ = \bigoplus_{k\in \Z} \cJ(\omega_X,\fra^{\lambda-k-1-\varepsilon}) u^k du,\]
which proves the claim.
\end{proof}

Finally, we define and study the Bernstein-Sato polynomials of the pair $(X,\fra)$. To do this, we recall the notation from the end of Section \ref{sec-prelim} above: if $Z\subseteq Y$ is a closed embedding of smooth varieties and $(\cM,F)$ is a filtered right $\cD_Y$-module, we define $G^\bullet_Z \cM = (F_{p(\cM)}\cM)\cdot V^\bullet_Z \cD_Y$. Then we define the $b$-function $b_{\cM,Z}(s)$ to be the minimal polynomial of the action of $\theta \in V^0\cD_Y \cap {\rm Der}_\C(\cO_Y)$ (which acts by the identity on $\cI_Z/\cI_Z^2$) on ${\rm Gr}_{G_Z}^0 \cM$.

To put ourselves into the above situation, we first assume $X\subseteq Y$ is some closed embedding into a smooth variety such that there exist $g_1,\dots, g_r \in \cO_Y(Y)$ which restrict to generators of $\fra \subseteq \cO_X$. Let $(\cM,F)$ denote the right $\cD_Y$-module underlying ${\rm IC}_X$.

Then we consider the graph embedding $\Gamma_g \colon Y \to Y \times \A^r_t$. In this situation, we will see that $b_{(X,\fra)}(s) = b_{Y \times \{0\},\Gamma_{g+}\cM}(s)$ is a good definition of Bernstein-Sato polynomial. However, this requires checking that the definition doesn't depend on choices. Once that is done, we can define a global version, too. Along the way, we verify the claim that $F_{-\dim X} V^\lambda(-)$ is well-defined on ${\rm IC}_X$, as claimed in Remark \ref{rmk-wellDefined} above and evidenced by \theoremref{thm-main}.

The following lemma shows that, with $Y$ fixed, the definition is unambiguous.

\begin{lem} \label{lem-unambiguous} If $(g_1^{(1)},\dots, g_{r_1}^{(1)}) = (g_1^{(2)},\dots, g_{r_2}^{(2)}) + \cI_X$ as ideals in $\cO_Y$, we have canonical identifications
\[ F_{-\dim X} V^\lambda  \Gamma_{g^{(1)},+}\cM \cong F_{-\dim X} V^\lambda \Gamma_{g^{(2)},+} \cM,\]
\[ b_{Y\times \{0\}, \Gamma_{g^{(1)},+}\cM}(s) = b_{Y\times \{0\}, \Gamma_{g^{(2)},+}\cM}(s).\]
 \end{lem}
\begin{proof} Without loss of generality, it suffices to compare the result for $g_1^{(2)},\dots, g_{r_2}^{(2)}$ with $g_1^{(2)},\dots, g_{r_2}^{(2)}, g_1^{(1)},\dots, g_{r_1}^{(1)}$.

First of all, write $\cI_X = (a_1,\dots, a_s) \subseteq \cO_Y$. Consider the graph embedding
\[ \Gamma_a \colon Y \to Y \times \A^s_u, y \mapsto (y, a_1(y),\dots, a_s(y)),\]
with coordinates $u_1,\dots, u_s$ on the $\A^s_u$ factor.

As $(g_1^{(1)},\dots, g_{r_1}^{(1)}) = (g_1^{(2)},\dots, g_{r_2}^{(2)}) + \cI_X$, we can write
\[ g_i^{(1)} = \sum c_{ij} g_j^{(2)} + \sum_{\ell=1}^s \chi_\ell a_\ell\]
for some $c_{ij}, \chi_\ell \in \cO_Y$. We consider the sequence of graph embeddings
\[ Y \times \A^s_u \to Y \times \A^s_u \times \A^{r_2}_t \to Y \times \A^s_u \times \A^{r_2}_t \times \A^{r_1}_\tau\]
\[ (y,u) \mapsto (y,u,g^{(2)}), \quad (y,u,t) \mapsto (y,u,t,\sum c t + \sum \chi u).\]

In fact, we have the commutative diagram
\[ \begin{tikzcd} Y \ar[r,"\Gamma_{g^{(2)}}"] \ar[d,"\Gamma_{a}"] & Y\times \A^{r_2}_t \ar[r,"\eta"] \ar[d,"\Gamma_{a}"] & Y\times \A^{r_2}_t \times \A^{r_1}_\tau \ar[d,"\Gamma_{a}"] \\ Y \times \A^s_u \ar[r,"\Gamma_{g^{(2)}}"]  & Y\times \A^s_u \times \A^{r_2}_t \ar[r,"\Gamma"] & Y\times \A^s_u \times \A^{r_2}_t \times \A^{r_1}_\tau \end{tikzcd},\]
where every morphism is a graph embedding. Moreover, $\eta \circ \Gamma_{g^{(2)}}$ is the graph embedding along $g^{(2)},g^{(1)}$, which we denote by $\Gamma_{(g^{(2)}, g^{(1)})}$.

For ease of notation, write $\cN_1 = \Gamma_{g^{(2)},+} \cM$ and $\cN_2 = \Gamma_{(g^{(2)},g^{(1)}),+} \cM = \eta_+ \cN_1$. 

For $i=1,2$ we have, by \lemmaref{lem-sameFunctions}, an isomorphism
\[ F_{-\dim X}V^\lambda_{Y \times \{0\}}  \cN_i = F_{-\dim X}V^\lambda_{Y\times \A^s_u \times \{0\}} \Gamma_{a,+}\cN_i\]
and equality
\[ b_{Y\times \{0\},\cN_i}(s) = b_{Y\times \A^s_u \times \{0\},\Gamma_{a,+}\cN_i}(s).\]

Now, as ${\rm IC}_X$ is supported on $X \subseteq Y$, and $a_1,\dots, a_r$ vanish on $X$, we see that $\Gamma_{a,+}\cN_1$ is supported on $Y \times \{0\} \times \A^{r_2}_t$ and similarly $\Gamma_{a,+} \cN_2$ is supported on $Y\times \{0\} \times \A^{r_2}_t \times \A^{r_1}_\tau$. Thus, there exist $\cN_1', \cN_2'$ such that
\[ \sigma_{i,+}(\cN_i',F) \cong \Gamma_{a,+}(\cN_i,F) \text{ for } i = 1,2,\]
where $\sigma_1 \colon Y\times \{0\} \times \A^{r_2}_t \to Y\times \A^s_u \times \A^{r_2}_t$ is the zero section and similarly $\sigma_2 \colon Y\times \{0\} \times \A^{r_2}_t \times \A^{r_1}_\tau \to Y\times \A^s_u \times \A^{r_2}_t \times \A^{r_1}_\tau$ is the zero section.

As this is an isomorphism, we clearly have identifications
\[F_{-\dim X}V^\lambda_{Y\times \A^s_u \times \{0\}} \Gamma_{a,+}\cN_i \cong F_{-\dim X}V^\lambda_{Y\times \A^s_u \times \{0\}}\sigma_{i,+}\cN_i'\]
and
\[ b_{Y\times \A^s_u \times \{0\},\Gamma_{a,+}\cN_i}(s) =  b_{Y\times \A^s_u \times \{0\},\sigma_{i,+}\cN_i}(s).\]

By \lemmaref{lem-sameFunctions}, we have identifications
\[F_{-\dim X}V^\lambda_{Y\times \A^s_u \times \{0\}} \sigma_{i,+}\cN_i' \cong V^\lambda_{Y \times \{0\}} \cN_i'\]
and
\[ b_{Y\times \A^s_u \times \{0\},\sigma_{i,+}\cN_i}(s) \cong b_{Y \times \{0\},\cN_i}(s).\]

But then by \lemmaref{lem-sameSubvariety}, we can identify the right hand side as follows:
\[F_{-\dim X}V^\lambda_{Y \times \{0\}} \cN_i' \cong F_{-\dim X}V^\lambda_{Y \times \{0\}} \sigma_{i,+}\cN_i'\]
and
\[ b_{Y \times \{0\},\cN_i}(s)= b_{Y \times \{0\},\sigma_{i,+}\cN_i}(s) .\]

We can replace $\sigma_{i,+}\cN_i$ in those identifications by $\Gamma_{a,+} \cN_i$. To summarize what we have shown so far, it is isomorphisms
\[ F_{-\dim X} V^\lambda_{Y\times \{0\}} \cN_i \cong F_{-\dim X} V^{\lambda}_{Y\times \{0\}}\Gamma_{a,+} \cN_i\]
and equalities
\[ b_{Y\times \{0\},\cN_i}(s) = b_{Y\times \{0\},\Gamma_{a,+} \cN_i}(s).\]

At this point we are done: indeed, note that $\Gamma_{a,+} \cN_2 = \Gamma_+ \Gamma_{a,+} \cN_1$, and hence we have by \lemmaref{lem-sameSubvariety}
\[F_{-\dim X}V^\lambda_{Y \times \{0\}} \Gamma_{a,+} \cN_1 =F_{-\dim X}V^\lambda_{Y \times \{0\}} \Gamma_{a,+} \cN_2 \]
and
\[ b_{Y \times \{0\},\Gamma_{a+} \cN_1}(s)= b_{Y \times \{0\},\Gamma_{a+} \cN_2}(s) .\]
\end{proof}

We prove an analogue of \cite{BMS}*{Thm. 5} in our setting. The proof is identical. 

Now, let $X_1,X_2$ be two possibly singular varieties with $X_i \hookrightarrow Y_i$. Let $F_1^{(i)},\dots, F_{r_i}^{(i)} \in \cO_{Y_i}$ restrict to $f_1^{(i)},\dots, f_{r_i}^{(i)}$ on $X_i$. Consider $F_1^{(1)},\dots, F_{r_1}^{(1)}, F_1^{(2)},\dots, F_{r_2}^{(2)} \in \cO_{Y_1\times Y_2}$ which restrict to $f_1^{(1)},\dots, f_{r_1}^{(1)}, f_1^{(2)},\dots, f_{r_2}^{(2)}$ on $X_1\times X_2$. Let $\cM_{Y_i}$ be the filtered right $\cD_{Y_i}$-module underlying the pushforward of ${\rm IC}_{X_i}^H$ on $Y_i$ and let $\cM$ be the filtered right $\cD_{Y_1\times Y_2}$-module underlying the pushforward of ${\rm IC}_{X_1\times X_2}^H$.

\begin{lem} \label{lem-Product} In the above notation, write $b_{Y_i\times \{0\}, \Gamma_{f^{(i)},+}\cM_{Y_i}}(s) = \prod_{\gamma \in \Q} (s+\gamma)^{m^{(i)}_{\gamma}}$. Then
\[ b_{Y_1 \times Y_2 \times \{0\}, \Gamma_{(f^{(1)},f^{(2)}),+}\cM}(s) = \prod (s+\gamma)^{m_\gamma},\]
where $m_\gamma = \max\{m_\alpha^{(1)} + m_{\beta}^{(2)} -1 \mid m_\alpha^{(1)} >0,\, m_{\beta}^{(2)} > 0, \,  \alpha+\beta = \gamma\}$.
\end{lem}
\begin{proof}
We consider $M_i = \Gamma_{F^{(i)},+} {\rm IC}_{X_i}$ on $Y_i \times \A^{r_i}_{t^{(i)}}$ and $\Gamma_{(F^{(1)},F^{(2)}),+} {\rm IC}_{X_1 \times X_2}$ on $Y_1 \times Y_2 \times \A^{r_1}_{t^{(1)}} \times \A^{r_2}_{t^{(2)}}$. We have the identification
\begin{equation} \label{eq-BoxTimes} \Gamma_{(F^{(1)},F^{(2)}),+} {\rm IC}_{X_1 \times X_2} = \Gamma_{F^{(1)},+} {\rm IC}_{X_1} \boxtimes \Gamma_{F^{(2)},+} {\rm IC}_{X_2} \end{equation}
using the isomorphism
\[ {\rm IC}_{X_1 \times X_2} = {\rm IC}_{X_1} \boxtimes {\rm IC}_{X_2}.\]

The isomorphism \ref{eq-BoxTimes} gives the isomorphism
\begin{equation} \label{eq-BoxTimesHodge} F_{-\dim X_1 -\dim X_2} \Gamma_{(F^{(1)},F^{(2)}),+} {\rm IC}_{X_1 \times X_2} = F_{-\dim X_1} \Gamma_{F^{(1)},+} {\rm IC}_{X_1} \boxtimes F_{-\dim X_2} \Gamma_{F^{(2)},+} {\rm IC}_{X_2} \end{equation}

Then by multiplication by $V^\bullet \cD_{Y_1 \times Y_2 \times \A^{r_1}_{t^{(1)}} \times \A^{r_2}_{t^{(2)}}}$ on both sides, we get the equality
\[ G^\bullet \Gamma_{(F^{(1)},F^{(2)}),+} {\rm IC}_{X_1 \times X_2}  = \sum_{i+j = \bullet} G^i \Gamma_{F^{(1)},+} {\rm IC}_{X_1}\boxtimes G^j\Gamma_{F^{(2)},+} {\rm IC}_{X_2}\]
and hence
\[{\rm Gr}_G^0\Gamma_{(F^{(1)},F^{(2)}),+} {\rm IC}_{X_1 \times X_2} = \bigoplus_{i+j = 0} {\rm Gr}_G^i \Gamma_{F^{(1)},+} {\rm IC}_{X_1}\boxtimes {\rm Gr}_G^j\Gamma_{F^{(2)},+} {\rm IC}_{X_2}.\]

Let $b(s)$ be the polynomial as defined in the lemma statement. Recall that ${\rm Gr}_G^{\ell_i} \Gamma_{F^{(i)},+} {\rm IC}_{X_i}$ is annihilated by $b_{f^{(i)}}(\theta_{t^{(i)}}+\ell_i)$. So for any $\ell_1,\ell_2$ we see that $b(\theta_{t^{(1)}} + \theta_{t^{(2)}} + \ell_1+\ell_2)$ annihilates the box product ${\rm Gr}_G^{\ell_1} \Gamma_{F^{(1)},+} {\rm IC}_{X_1} \boxtimes {\rm Gr}_G^{\ell_2}\Gamma_{F^{(2)},+} {\rm IC}_{X_2}$. We conclude that $b(\theta_{t^{(1)}} + \theta_{t^{(2)}})$ annihilates ${\rm Gr}_G^0 \Gamma_{(F^{(1)},F^{(2)}),+} {\rm IC}_{X_1 \times X_2}$, and so we get (by minimality) $b_{(f^{(1)},f^{(2)})}(s) \mid b(s)$.

By the above observation, however, $b(s)$ is precisely the minimal polynomial of $\theta_{t^{(1)}} +\theta_{t^{(2)}}$ on ${\rm Gr}_G^0\Gamma_{F^{(1)},+} {\rm IC}_{X_1} \boxtimes {\rm Gr}_G^0\Gamma_{F^{(2)},+} {\rm IC}_{X_2}$, which proves the other divisibility relation and hence the claimed equality.
\end{proof}

With these lemmas, we can prove independence of the smooth ambient variety $Y$.

\begin{lem} \label{lem-ChangeY} Let $Y_1,Y_2$ be two smooth varieties containing $X$ and let $f_1^{(i)},\dots, f_r^{(i)}$ be functions on $Y_i$ such that $f_j^{(i)}\vert_X = f_j$ for $1\leq j \leq r$. We write $F_{-\dim X} V^\bullet_{Y_i} {\rm IC}_X^H$ for the $V$-filtration on $F_{-\dim X} {\rm IC}_X^H$ induced by the graph embedding of $f_1^{(i)},\dots, f_{r_i}^{(i)}$ on $Y_i$. Then we have an isomorphism
\[ F_{-\dim X} V^\bullet_{Y_1} {\rm IC}_X^H \cong F_{-\dim X} V^\bullet_{Y_2} {\rm IC}_X^H\]
and an equality
\[ b_{Y_1\times \{0\}, \Gamma_{f^{(1)},+}\cM_{Y_1}}(s) = b_{Y_2\times \{0\}, \Gamma_{f^{(2)},+}\cM_{Y_2}}(s).\]
\end{lem}
\begin{proof} 

Let $X \to Y_1 \times Y_2$ be the diagonal embedding. Let $F_j^{(i)} = \pi_i^*(f_j^{(i)}) \in \cO_{Y_1\times Y_2}(Y_1\times Y_2)$. For any ideal $\frb \subseteq \cO_{Y_1\times Y_2}$, we write $F_{-\dim X} V^\bullet_{Y_1\times Y_2, \frb} {\rm IC}_X^H$ for the $V$-filtration on $F_{-\dim X} {\rm IC}_X^H$ induced by the graph embedding along some set of generators of $\frb$ on $Y_1 \times Y_2$ (by \lemmaref{lem-unambiguous}, this is independent of the choice of generators).

By \lemmaref{lem-unambiguous}, we see that the invariants for $(F_1^{(1)},\dots, F_{r}^{(1)})$ and $(F_1^{(2)},\dots, F_{r}^{(2)})$ are the same. In fact, as the ideal they span is unchanged if we append $0$'s, we see that the invariants for $(F_1^{(1)},\dots, F_{r}^{(1)},0,\dots, 0)$ and $(0,\dots, 0, F_1^{(2)},\dots, F_{r}^{(2)})$ are the same.

For the claim about $b$-functions, it suffices to prove that the $b$-function for $f_1^{(1)},\dots, f_{r}^{(1)}$ is the same as that for $F_1^{(1)},\dots, F_{r}^{(1)},0,\dots, 0$. This follows immediately from \lemmaref{lem-Product}, where the second set of functions are taken to be $0$ (where we apply \ref{eg-ZeroEmbedding} to compute that the $b$-function is simply $s$ in this case).

For the claim about $V$-filtrations, it suffices to prove
\begin{equation} \label{eq-PullbackProduct} F_{-\dim X} V^\bullet_{Y_1} {\rm IC}_X^H = F_{-\dim X} V^\bullet_{Y_1\times Y_2, (F^{(1)})} {\rm IC}_X^H.\end{equation}

For this, 
\[ \begin{tikzcd} X\ar[r] \ar[dr] & Y_1\times Y_2 \ar[r,"\Gamma_{F^{(1)}}"] \ar[d] & Y_1\times Y_2 \times \A^{r}_{t^{(1)}} \ar[d,"\Pi"] \\ {} & Y_1 \ar[r,"\Gamma_{f^{(1)}}"] & Y_1\times \A^{r}_{t^{(1)}}\end{tikzcd},\]
so that $\Pi_* \Gamma_{F^{(1)},+} {\rm IC}_X^H = \Gamma_{f^{(1)},*} {\rm IC}_X^H$. Here the vertical maps are the projections. Note that $\Pi$ is an isomorphism (hence projective) on the support of $\Gamma_{F^{(1)},+}{\rm IC}_X^H$, which is the graph of $f_1^{(1)},\dots, f_{r}^{(1)}$ on $X$. By Saito's bi-strictness (for $r = 1$ given by \cite{SaitoMHP}*{3.3.17} and for $r > 1$ given by \propositionref{prop-bistrictness} above), we conclude that the equality \ref{eq-PullbackProduct} holds, which finishes the proof.
\end{proof}

We can now give the definition of the Bernstein-Sato polynomial:
\begin{defi} \label{defi-BS} Let $X$ be a connected variety and let $\fra \subseteq \cO_X$ be an ideal sheaf. Let $X = \bigcup_{i\in I} U_i$ be an open cover such that $U_i$ can be embedded into a smooth variety $Y_i$ so that there exist $f_1^{(i)},\dots, f_{r_i}^{(i)} \in \cO_{Y_i}(Y_i)$ which restrict to generators of $\fra\vert_{U_i}$.

Define
\[ b_{(X,\fra)}(s) = {\rm lcm}_{i\in I} b_{Y_i \times \{0\}, \Gamma_{f^{(i)},+} \cM_{Y_i}}(s),\]
where $(\cM_{Y_i},F)$ is the right filtered $\cD_{Y_i}$-module underlying ${\rm IC}_{U_i}$.
\end{defi}

Such a cover always exists: take any affine open cover. It is easy to check, by the above lemmas, that this is independent of the cover. Moreover, if $U\subseteq X$ is an open subset, then $b_{(U,\fra\vert_U)}(s) \mid b_{(X,\fra)}(s)$. If $X$ were already smooth, this definition agrees with that in \cite{BMS}. Hence, we conclude the following: if $\fra \vert_{X_{\rm reg}} = (\widetilde{f}_1,\dots, \widetilde{f}_r)$, then
\[ b_{\widetilde{f}}(s) \mid b_{(X,\fra)}(s),\]
where the left hand side is the Bernstein-Sato polynomial as defined in \cite{BMS}.

We verify the other properties of $b_{(X,\fra)}(s)$ claimed in \theoremref{thm-bFunction}:
\begin{proof}[Proof of \theoremref{thm-bFunction}] The claims are local, so throughout, fix an embedding $X \hookrightarrow Y$ with $F_1,\dots, F_r$ on $Y$ restricting to $f_1,\dots, f_r$ on $X$. Let $(\cM,F)$ be the underlying right filtered $\cD_Y$-module for ${\rm IC}_X^H$ and let $(\cB_{f,X},F) = \Gamma_+(\cM,F)$ where $\Gamma \colon Y \to Y\times \A^r_t$ is the graph embedding along $F_1,\dots, F_r$.

For the proof of (\ref{itm-LCM}), by definition we have a decomposition ${\rm IC}_X^H = \bigoplus_{i\in I} {\rm IC}_{X_i}^H$. This gives a decomposition of the underlying filtered $\cD$-modules (hence, also of the $G$-filtered $\cD$-modules after applying the graph embedding), so the claim is clear.

For the proof of (\ref{itm-NegRoots}), it suffices by the previous property to assume that $X$ is irreducible. If $\fra$ is the zero ideal, then clearly $b_{(X,\fra)}(s) = s$. So we can assume now that $\fra$ is not the zero ideal. Note that $\cB_{f,X}$ has strict support equal to $\Gamma(X)$.

As $\fra$ is non-zero, we know that $\Gamma(X)$ is not contained in the zero section of $Y \times \A^r_t$. Thus, we see that $\cB_{f,X}$ has no quotient objects supported on $\{t_1 = \dots = t_r = 0\}$. Thus, $F_{-\dim X} \cB_{f,X} \subseteq V^{>0}\cB_{f,X}$. Indeed, if this weren't true, then $F_{-\dim X} {\rm Gr}_V^\mu \cB_{f,X}$ would be non-zero for some $\mu \leq 0$. We have the filtered surjectivity of the map
\[ \bigoplus_{i=1}^r F_{q-1} {\rm Gr}_V^{\mu+1} \cB_{f,X} \xrightarrow[]{t_i} F_q {\rm Gr}_V^{\mu} \cB_{f,X}\]
for all $\mu \leq 0$ (by \cite{CD}*{Thm. A} for $\mu < 0$ and \cite{CD}*{Cor. 3.4} for $\mu = 0$, combined with the strictness of that morphism ensured by \cite{CD}*{Thm. 1.2}). But then as $F_{-\dim X -1} \cB_{f,X} = 0$, this proves $F_{-\dim X} {\rm Gr}_V^{\mu}\cB_{f,X} = 0$ for all $\mu \leq 0$, a contradiction.

In particular, $G^0 \cB_{f,X} \subseteq V^{>0}\cB_f$. Thus, for all $\mu \leq 0$ we have ${\rm Gr}_V^\mu {\rm Gr}_G^0 \cB_{f,X} = 0$, and so $(s+\mu) \nmid b_f(s)$.

For the proof of (\ref{itm-JumpingNumbers}), by definition, the smallest root is the minimal $\lambda$ such that ${\rm Gr}_V^\lambda {\rm Gr}_G^0(\cB_{f,X}) \neq 0$. By \theoremref{thm-main}, we see that ${\rm lct}(\omega_X,\fra)$ is the maximal $\lambda$ such that $F_{-\dim X} \cM \delta_f \subseteq V^\lambda \cB_{f,X}$, or equivalently, the minimal $\lambda$ such that $F_{-\dim X} \cB_{f,X} \not \subseteq V^{>\lambda}\cB_{f,X}$. Using stability by $V^0 \cD_{Y\times \A^r_t}$, this is equivalent to the minimal $\lambda$ such that $(F_{-\dim X} \cB_{f,X}) \cdot V^0 \cD_{Y\times \A^r_t} = G^0 \cB_{f,X} \not \subseteq V^{>\lambda}\cB_{f,X}$.

Assume $\lambda$ is such that $G^0 \cB_{f,X} \subseteq V^\lambda \cB_{f,X}$ and $G^0\cB_{f,X} \not \subseteq V^{>\lambda}\cB_{f,X}$. Then
\[ {\rm Gr}_V^\lambda {\rm Gr}_G^0 \cB_{f,X} = \frac{G^0 V^\lambda\cB_{f,X}}{G^0 V^{>\lambda}\cB_{f,X} + G^1 V^\lambda\cB_{f,X}} = \frac{G^0\cB_{f,X}}{G^0 V^{>\lambda}\cB_{f,X} + G^1\cB_{f,X}},\]
where we used that $G^1\cB_{f,X}\subseteq G^0\cB_{f,X} \subseteq V^\lambda\cB_{f,X}$. Note that $G^1\cB_{f,X} = \sum_{i=1}^r (G^0 \cB_{f,X})t_i \subseteq V^{\lambda +1}\cB_{f,X} \subseteq V^{>\lambda}\cB_{f,X}$, so this quotient is simply
\[ {\rm Gr}_V^\lambda {\rm Gr}_G^0\cB_{f,X} = G^0 {\rm Gr}_V^\lambda \cB_{f,X},\]
which is non-zero by choice of $\lambda$. Thus, for such a $\lambda$, we conclude $s+\lambda \mid b_f(s)$.

We claim that this $\lambda$ is minimal with that property. Let $\mu < \lambda$ be such that $s+\mu \mid b_f(s)$. Then ${\rm Gr}_V^\mu {\rm Gr}_G^0 \cB_{f,X} \neq 0$. However, we know $G^0 \cB_{f,X} \subseteq V^\lambda \cB_{f,X} \subseteq V^{\mu}\cB_{f,X}$, and so the quotient ${\rm Gr}_V^{\mu} {\rm Gr}_G^0 \cB_{f,X} = 0$, a contradiction. Thus, there are no $\mu < \lambda$ with $s+ \mu \mid b_f(s)$, as claimed.

Finally, let $\beta \in [\lambda,\lambda+1)$, where $\lambda$ is the same as above (by definition, $\lambda = {\rm lct}(\omega_X,\fra)$). The equality
\[ G^1 \cB_{f,X} = \sum_{i=1}^r G^0\cB_{f,X} t_i \text{ implies } G^1 \cB_{f,X} \subseteq V^{\lambda+1}\cB_{f,X} \subseteq V^{>\beta}\cB_{f,X},\]
and so we see that
\[ {\rm Gr}_V^\beta {\rm Gr}_G^0(\cB_{f,X}) = \frac{ G^0 V^\beta \cB_{f,X}}{G^0 V^{>\beta}\cB_{f,X} + G^1 V^\beta \cB_{f,X}} = \frac{G^0 V^\beta \cB_{f,X}}{G^0 V^{>\beta}\cB_{f,X}} = G^0 {\rm Gr}_V^{\beta}(\cB_{f,X}).\]

If $\beta$ is such that $\cJ(\omega_X,\fra^{\beta-\varepsilon}) \neq \cJ(\omega_X,\fra^{\beta})$ for $0 < \varepsilon \ll 1$, then we have proper containment $F_{-\dim X} V^{\beta} \cB_{f,X} \supsetneq F_{-\dim X} V^{>\beta} \cB_{f,X}$. Hence, there exists an element $m \in F_{-\dim X} V^\beta \cB_{f,X}$ which doesn't lie in $V^{>\beta}\cB_{f,X}$, and so, the containment $F_{-\dim X} V^\beta \cB_{f,X} \subseteq G^0 V^\beta \cB_{f,X}$ implies that we have non-vanishing $G^0 {\rm Gr}_V^\beta(\cB_{f,X}) \neq 0$, hence $s+\beta \mid b_{(X,\fra)}(s)$, as claimed.
\end{proof}

\begin{proof}[Proof of \theoremref{thm-BSRHM}] If $X$ is a rational homology manifold, then $c(X) = +\infty$, hence, by \lemmaref{lem-adjunction}, we also have $c(Z) = +\infty$. Thus, $\Q_Z^H[\dim Z]$ is a single Hodge module on $Z$.

To prove the claim, it suffices to show that $\mathbf D(\Q_Z^H[\dim Z])$ is a pure Hodge module. If $i\colon Z \to X$ is the closed embedding, we have by the rational homology manifold property
\[ i^! {\rm IC}_X^H = i^! \mathbf D(\Q_X^H[\dim X]) = \mathbf D(\Q_Z^H[\dim Z])[-r].\]

As the question is local, we can assume $\iota \colon X\subseteq Y$ with $Y$ smooth and $f_1,\dots, f_r \in \cO_Y(Y)$. Let $\Gamma \colon Y \to Y \times \A^r_t$ be the graph embedding along $f_1,\dots, f_r$. Then by base change we have
\[ \sigma^! \Gamma_+ \iota_* {\rm IC}_X^H = \iota_* i_* i^! {\rm IC}_X^H = \iota_* i_* \mathbf D(\Q_Z^H[\dim Z])[-r].\]

Thus, to prove the desired purity, it suffices by \cite{CD}*{Thm. B} to prove that the monodromy filtration on ${\rm Gr}_V^r(\Gamma_+ \cM)$ is trivial, or in other words, that $N = 0$ on ${\rm Gr}_V^r (\Gamma_+ \cM)$. Here $N = \sum_{i=1}^r t_i \de_{t_i} +r$ acting on the right and $\cM$ is the right $\cD_Y$-module underlying $\iota_*{\rm IC}_X^H$.

To check this, we use exhaustiveness of the $G$-filtration $G^\bullet \Gamma_+ \cM$ as defined above. This implies that $G$ is exhaustive on ${\rm Gr}_V^r(\Gamma_+ \cM)$, too. Note also that $G^0 \Gamma_+ \cM \subseteq V^{>0}\Gamma_+ \cM$ implies that $G^r \Gamma_+ \cM = \sum_{|\alpha| =r} (G^0 \Gamma_+ \cM)t^\alpha \subseteq V^{>r} \Gamma_+ \cM$, which gives
\[ G^r {\rm Gr}_V^r (\Gamma_+ \cM) = 0.\]

For any $k\geq 0$, we have (by definition of the $V$-filtration on $\cD_{Y \times \A^r_t}$) surjections 
\[\bigoplus_{|\alpha| = k} {\rm Gr}_G^0(\Gamma_+ \cM) \xrightarrow[]{t^\alpha} {\rm Gr}_G^k(\Gamma_+ \cM)\]
\[(\text{resp. } \bigoplus_{|\alpha| = k} {\rm Gr}_G^0(\Gamma_+ \cM) \xrightarrow[]{\de_t^\alpha} {\rm Gr}_G^{-k}(\Gamma_+ \cM)),\]
so that $\theta_t = \sum_{i=1}^r t_i \de_{t_i}$ on the left gets sent to $\theta_t + k$ (resp. $\theta_t - k$) on the right.

From those surjections, we see that if $b_k(s)$ is the minimal polynomial of the action of $\theta_t =\sum_{i=1}^r t_i \de_{t_i}$ on ${\rm Gr}_G^k(\Gamma_+ \cM)$, then $b_k(s) \mid b_0(s+k) = b_{(X,\fra)}(s+k)$ for all $k\in \Z$. As we assume $-r$ is the only integer root of $b_{(X,\fra)}(s)$, we see that $-r-k$ is the only candidate integer root of $b_k(s)$ for all $k \in \Z$. Hence, linear algebra tells us that
\[ {\rm Gr}_V^j {\rm Gr}_G^k(\Gamma_+ \cM) = 0 \text{ unless } j = k+r.\]

In particular, we have equality ${\rm Gr}_V^r(\Gamma_+ \cM) = {\rm Gr}_G^0 {\rm Gr}_V^r (\Gamma_+ \cM)$. By the assumption that $-r$ is a simple root of $b_{(X,\fra)}(s)$, we only need a single power of $\theta+r$ to annihilate the right hand side, which proves the claim. The same proof shows that, for any $j\in \Z$, the monodromy filtration on ${\rm Gr}_V^j(\Gamma_+ \cM)$ is trivial. 
\end{proof}

\section{Toward examples} \label{sec-examples}
It is non-trivial to compute the Bernstein-Sato polynomial even when $X$ is smooth. In this section, we provide some comments that assist in the calculation of $b_{(X,\fra)}(s)$ in some special cases. It is convenient in this section to work with left $\cD$-modules.

\begin{eg} \label{eg-Node} If $X$ is a reduced SNC divisor in some ambient smooth variety, with $X = \bigcup X_i$ the union of the smooth irreducible divisors $X_i$ giving the irreducible components, then Property \ref{itm-LCM} of \theoremref{thm-bFunction} implies that
\[ b_{(X,\fra)}(s) = {\rm lcm}_i b_{(X_i,\fra\vert_{X_i})}(s),\]
and so this provides many examples of Bernstein-Sato polynomials on a singular ambient variety, by reducing to the calculation of usual Bernstein-Sato polynomials on ambient smooth varieties.

For example, if $X = V(xy) \subseteq \A^2$ and $f = x$, we get
\[ b_{(X,\fra)}(s) = s(s+1).\]
\end{eg}

We proceed to describe an example with $X$ irreducible.

Assume $X \subseteq \A^n$ is reduced and irreducible, defined by $g_1,\dots, g_b \in \C[x_1,\dots, x_n] = S$ forming a regular sequence, where for each $1\leq i\leq b$ we assume $g_i$ is weighted homogeneous of degree $d \geq 2$ and ${\rm wt}(x_j) = w_j \in \Z_{\geq 1}$. Let $\theta_w = \sum_{i=1}^n w_i x_i \de_{x_i}$ be the weighted Euler operator, so that $\theta_x(g_i) = d g_i$ for all $1\leq i\leq r$. Assume moreover that $X$ has an isolated singularity at the origin.

Let $f \in \C[x_1,\dots, x_n]$ be weighted homogeneous (in the same weights) of degree $d_f$, defining the ideal $\fra = (f)$. We assume moreover that $f$ has (at worst) an isolated singularity at the origin, and is smooth otherwise.

Let $\cM$ be the left $\cD_{\A^n}$-module associated to ${\rm IC}_X$. As $g_1,\dots, g_b$ form a regular sequence, we see that
\[ \cM \subseteq \cH^b_{g}(S).\]

In particular, if we apply $\Gamma_+$ (where $\Gamma \colon \A^n \to \A^{n+1}$ is the graph embedding along $f$), we have containment
\[ \Gamma_+ \cM \subseteq \Gamma_+ \cH^b_{g}(S),\]
and so we have equality for any $\lambda \in \Q$:
\[ V^\lambda \Gamma_+ \cM = (\Gamma_+ \cM) \cap V^{\lambda}\Gamma_+ \cH^b_g(S).\]

In particular, the jumping numbers of the $V$-filtration of $\Gamma_+ \cM$ are contained in the jumping numbers of the $V$-filtration of $\Gamma_+ \cH^b_g(S)$.

\textbf{Claim} For all $\lambda \notin \Z_{\geq 1}$, the modules
\[ {\rm Gr}_V^\lambda (\Gamma_+ \cM) \subseteq {\rm Gr}_V^\lambda (\Gamma_+ \cH^b_g(S))\]
are $\cD_{\A^n}$-modules supported on the origin $0 \in \A^n$.

\textit{Proof of Claim} Away from the origin, $X$ is smooth and $f$ defines a smooth function, so this follows from the usual computation in the smooth case.

An arbitrary element of $\Gamma_+ \cH^b_g(S) = \bigoplus_{k \geq 0} \cH^b_g(S) \de_t^k \delta_f$ can be written as a finite sum of elements of the form $\frac{h}{g^\alpha} \de_t^k \delta_f$ where $h \in S$ is weighted homogeneous. An easy computation gives
\begin{equation} \label{eq-EulerAction} \theta_w( \frac{h}{g^\alpha} \de_t^k\delta_f) = (\deg(h) - d|\alpha| +d_f(s-k)) \frac{h}{g^\alpha} \de_t^k \delta_f\end{equation}
where $s = -\de_t t$. 

\begin{lem} \label{lem-denominatorWeightedHom} For $h\in S$ weighted homogeneous and fixed, if $\frac{h}{g^\alpha} \de_t^k \delta_f \notin V^1 \Gamma_+ \cH^b_g(S)$, then $\frac{h}{g^\alpha} \de_t^k \delta_f$ defines a non-zero element in ${\rm Gr}_V^{\lambda}(\Gamma_+ \cH^b_g(S))$ for some $\lambda \in (\frac{1}{d_f} \Z) \cap \Q_{ < 1}$. Moreover, $(s+\lambda) \cdot {\rm Gr}_V^\lambda(\cH^b_g(S)) = 0$ for all $\lambda \notin \Z_{\geq 1}$.

In particular, we have 
\[ \{\text{Jumping numbers of } V^\bullet \Gamma_+ \cM\}= \{\gamma \mid b_{(X,\fra)}(-\gamma) =0\} + \Z \subseteq \frac{1}{d_f} \Z.\]
\end{lem}
\begin{proof} The equality in the final statement is obvious, and the containment follows from the fact that the jumping numbers of $V^\bullet \Gamma_+ \cM$ are contained in those of $V^\bullet \Gamma_+ \cH^b_g(S)$, and the latter is contained in $\frac{1}{d_f} \Z$ by the first statement, using that every element is a sum of weighted homogeneous elements, and that the jumping numbers which are strictly less than 1 completely determine all jumping numbers.

Assume $\frac{h}{g^\alpha}\de_t^k \delta_f \notin V^1 \Gamma_+ \cH^b_g(S)$, so that it defines a non-zero element in ${\rm Gr}_V^\lambda(\Gamma_+ \cH^b_g(S))$ for some $\lambda \in \Q_{<1}$. By the claim above, this module is supported on $0 \in \A^n$, so by Kashiwara's equivalence, it is of the form $\bigoplus_{\beta \in \N^n} \cN \de_x^\beta \delta_0$ for some $\cD_{\{0\}}$-module $\cN$. Here $x_j(n \delta_0) = 0$ for all $1\leq j \leq n$, and the rest of the $\cD_{\A^n}$-action is given by the Leibniz rule.

We know that $s+\lambda$ acts nilpotently on this module by definition. Hence, by the computation \ref{eq-EulerAction} above, the element $\frac{h}{g^\alpha}\de_t^k \delta_f$ is a generalized eigenvector for the operator $\theta_w + d|\alpha| - \deg(h) +d_f(\lambda+k)$. We have
\[ \theta_w ( n \de_x^\alpha \delta_0) = \de_x^\alpha((\theta_w - \alpha\cdot w) n \delta_0) = (-|w| - \alpha\cdot w) (n\de_x^\alpha \delta_0),\]
so that all generalized eigenvectors for $\theta$ are actually eigenvectors (this argument also shows that $s+\lambda$ acts by $0$ on ${\rm Gr}_V^{\lambda}(\Gamma_+ \cH^b_g(S))$ for $\lambda \notin \Z_{\geq 1}$, rather than just nilpotently) and the eigenvalues are necessarily of the form $-|w| -\alpha\cdot w$ for some $\alpha \in \N^n$.

Thus, we see that $d|\alpha| - \deg(h) + d_f(\lambda+k) \in \Z$, which gives $\lambda \in \frac{1}{d_f} \Z$ as claimed.
\end{proof}

\begin{eg} Perhaps the simplest nontrivial example is the hypersurface $X = \C[x,y,z]/(x^2+y^2+z^2)$ with $f = x$. As the zero locus of $f$ in $X$ is a nodal curve (defined in $\A^2_{y,z}$ by $y^2+z^2$), it is not a rational homology manifold. So we see by \theoremref{thm-BSRHM} (where we know that $X$ itself is a rational homology manifold) that either $-1$ is not a simple root of $b_{(X,\fra)}(s)$ or that there is another integer root of $b_{(X,\fra)}(s)$.

Moreover, as $x^2+y^2+z^2$ is homogeneous with an isolated singularity at the origin and $x$ is homogeneous of degree $1$, we see by \lemmaref{lem-denominatorWeightedHom} that the only roots of $b_{(X,\fra)}(s)$ are integers.

In fact, we claim that $b_{(X,\fra)}(s) = (s+1)^2$. To see this, note that because $x$ is already a smooth function, we don't have to use the graph embedding. Also, as $x^2+y^2+z^2$ is a rational homology manifold with rational singularities, we can compute the lowest Hodge piece of ${\rm IC}_X = \cH^1_g(S)$ as the first term of the pole order filtration (see \cite{Olano} or more generally \cite{CDM}*{Thm. 1.4}). In other words, the lowest Hodge piece is generated over $\C[x,y,z]$ by $\frac{1}{g}$

Using that $s = -\de_x x = -x\de_x -1$, we need to verify containment
\[ (x\de_x)^2(1/g) \in V^1 \cD_{\A^3}\cdot (1/g) = \C[x,y,z,x\de_x,\de_y,\de_z] \cdot \frac{x}{g}\]
in $\cH^1_g(S)$. It is simple to check that the left hand side is equal to
\[ \frac{-4gx^2 + 8x^4}{g^3}.\]

Similarly, we can check that
\[ (\de_y^2 + \de_z^2)(\frac{1}{g}) = \frac{-4g+8x^2}{g^3},\]
and so if we multiply by $x^2$, we get
\[ (x\de_x)^2(1/g) = \frac{4gx^2-8x^4}{g^3} = -x(\de_y^2 + \de_z^2)(\frac{x}{g}),\]
as claimed. So we know that the Bernstein-Sato polynomial divides $(s+1)(s+2)$, but by the discussion at the beginning of the example, we must have equality.
\end{eg}

\begin{eg} In general, if $X \subseteq \A^n$ is defined by $g = x_1^2+ \dots + x_n^2$ for $n\geq 3$, we can use the same relation above to give a divisibility relation for $\fra = (x_1)$. Once again, the only possible roots are integers because $x_1$ has degree 1. Also, although if $n$ is even $X$ is no longer a rational homology manifold, it always has rational singularities, and so once again the lowest Hodge piece of ${\rm IC}_X^H$ (viewed inside $\cH^1_g(\C[x_1,\dots ,x_n])$) is generated by $\frac{1}{g}$.

We have
\[ \de_{x_i}^2(\frac{1}{g}) = \de_{x_i}(\frac{-2x_i}{g^2}) = \frac{-2}{g^2} + \frac{8x_i^2}{g^3},\]
and so we have
\[ (\sum_{j=2}^n \de_{x_j}^2)(\frac{1}{g}) = \frac{(10-2n)g - 8 x_1^2}{g^3}.\]

On the other hand, we have $(x_1\de_{x_1})(\frac{1}{g}) = \frac{-2x_1^2}{g^2}$ and
\[ (x_1 \de_{x_1})^2(\frac{1}{g}) = \frac{8x_1^4 - 4gx_1^2}{g^3}.\]

We conclude 
\[ (x\de_x)(x\de_x - (n-3))(\frac{1}{g}) = \frac{8x_1^4 - (10-2n)g}{g^3} = -x_1 (\sum_{j=2}^n \de_{x_j}^2)(\frac{x_1}{g}).\]

So we get 
\[ (s+1) \mid b_{(X,\fra)}(s) \mid (s+1)(s + (n-2)).\]

If $n$ is odd, then $X$ is a rational homology manifold, but the vanishing locus of $x_1$ inside $X$ defines $x_2^2+ \dots + x_n^2$ which is not a rational homology manifold (the Bernstein-Sato polynomial is $(s+1)(s+\frac{n-1}{2})$ which has an extra integer root, and in case the ambient variety is smooth and $r=1$, the statement of \theoremref{thm-BSRHM} can be upgraded to an if and only if). So in this case, we know $b_{(X,\fra)}(s) = (s+1)(s+(n-2))$ by \theoremref{thm-BSRHM}.
\end{eg}

\begin{eg} Interestingly, if we take $f = x_1 -x_2$ in the example above (so $X$ is defined by $g = x_1^2+ \dots + x_n^2$ for $n\geq 3$), we once again get $b_{(X,\fra)}(s) \mid (s+1)(s+(n-2))$, with equality if $n$ is odd. Taking $n=2$ fits into this pattern as well: we can write $x_1^2+x_2^2 = (x_1 + i x_2)(x_1-ix_2)$ and we computed above that the $b$-function for $f = x_1 - i x_2$ should be $s(s+1)$.

To see this, work in the coordinates $x_1-x_2, x_2,x_3,\dots, x_n$, we have
\[ \de_{x_1} = \de_{x_1 -x_2}, \quad \de_{x_2} = -\de_{x_1-x_2} + \de_{\widehat{x}_2},\]
and so the vector field we consider is $(x_1-x_2) \de_{x_1}$ and the allowed operators for the functional equation are $\C[x_1,\dots, x_n, \de_{x_3},\dots, \de_{x_n},(\de_{x_1}+\de_{x_2}), (x_1-x_2)\de_{x_1}]$. For now, we write $\tau = (x_1-x_2)\de_{x_1}$. Then one can check that for any $\lambda \in \C$ we have equality
\[ (\tau)(\tau+\lambda)(\frac{1}{g}) = -2\frac{x_1-x_2}{g^2} \left(((2+\lambda)x_1-x_2) - \frac{4x_1^2(x_1-x_2)}{g} \right).\]

On the other hand, one can also check that the expression\[ \left(\frac{n-6}{2}(\de_{x_1}+\de_{x_2})-\frac{1}{2}(x_1+x_2)(\de_{x_1}^2+\de_{x_2}^2) -\frac{1}{2}(3x+y)(\sum_{j=3}^n \de_{x_j}^2)\right)(\frac{x_1-x_2}{g})\]
is equal to
\[ -2\frac{x_1-x_2}{g^2} \left(((5-n)x_1-x_2) - \frac{4x_1^2(x_1-x_2)}{g} \right)\]
and it lies in $V^1 \cD \cdot \frac{1}{g}$. Thus, if we take $\lambda$ such that $2+\lambda = 5-n$ (in other words, $\lambda = 3-n$), we get the desired functional equation. 
\end{eg}

\begin{eg} We give a similar computation for $X$ defined by $g = x^3+y^2$ in $\A^2_{x,y}$ and $f = x$. It would be interesting to see what the polynomial is for $f = y$, as we comment at the end of the example. We will see that
\[ b_{(X,x)}(s) = (s+1)(s+\frac{1}{2}),\]
where the denominator $2$ is expected by \lemmaref{lem-denominatorWeightedHom} above, as $x$ has weight $2$ in this case.

First of all, $X$ is a rational homology manifold, and so we have equality ${\rm IC}_X^H = \cH^1_g(\C[x,y])$. By definition, the lowest Hodge piece of local cohomology has numerators given by the $0$th Hodge ideal, which by \cite{MP1}*{Prop. 10.1} is the multiplier ideal $\cJ(\C[x,y], g^{1-\varepsilon}) = (x,y)$ (for $0 < \varepsilon \ll 1$). Once again, $f = x$ is a smooth function on $\C[x,y]$, so we do not need to use the graph embedding. Using the relation $x\de_x = \de_x x -1 = -(s+1)$, we see that we need to find some polynomial $b(s)$ such that
\[ b(x\de_x)(\frac{x}{g}) \in \C[x,y,\de_y] \cdot \frac{x^2}{g} + \C[x,y,\de_y]\cdot \frac{xy}{g},\]
\[ b(x\de_x)(\frac{y}{g}) \in \C[x,y,\de_y] \cdot \frac{x^2}{g} + \C[x,y,\de_y]\cdot \frac{xy}{g},\]
and then $b_{(X,\fra)}(s) \mid b(-s-1)$.

For the first one, we can compute
\[ ( x\de_x +\frac{1}{2}) ( \frac{x}{g}) = \frac{3xy^2}{g^2}-\frac{3}{2}\frac{x}{g} = -\frac{3}{2} \de_y(\frac{xy}{g}),\]
where the right hand side lies in the desired submodule. 

For the second one, we get
\[ (x\de_x)(\frac{y}{g}) = -3 \frac{x^3 y}{g^2} = \frac{3}{2} x \de_y(\frac{x^2}{g})\]
and again the right hand side lies where it should. So we conclude that $b(s) = s(s+\frac{1}{2})$ works, and thus,
\[ b_{(X,\fra)}(s) \mid b(-s-1) = (s+\frac{1}{2})(s+1).\]

Once again, by \theoremref{thm-BSRHM} we can conclude that equality holds.

For $f=y$, we have found the equality
\[ (y\de_y + \frac{1}{3})(\frac{y}{g}) = \frac{4}{3} \frac{y}{g} - 2\frac{y^3}{g^2} = -\frac{2}{3} \de_x(\frac{xy}{g}),\]
which lies where it should for the Bernstein-Sato polynomial. The relation for the element $\frac{x}{g}$ is quite difficult to work out. Once again, the denominator 3 is expected as $y$ has weight $3$ here.
\end{eg}

\end{document}